\documentclass[a4paper,12pt]{article}
\usepackage{amssymb,amsmath,bm}
\usepackage{graphicx}
\usepackage{enumerate}
\usepackage{mathrsfs}
\usepackage{amsthm}

\title{An Equivalence Result on the Order of Differentiability in Frobenius' Theorem}
\author{Yuhki Hosoya\thanks{E-mail: ukki(at)gs.econ.keio.ac.jp,\ ORCID ID:0000-0002-8581-4518}~{}\thanks{Faculty of Economics, Chuo University. 742-1, Higashinakano, Hachioji, Tokyo, 192-0393 Japan.}}

\makeatletter
\renewenvironment{proof}[1][\proofname]{\par
  \pushQED{\qed}%
  \normalfont \topsep6\p@\@plus6\p@\relax
  \trivlist
  \item\relax
  {\bfseries
  #1\@addpunct{.}}\hspace\labelsep\ignorespaces
}{%
  \popQED\endtrivlist\@endpefalse
}
\makeatother

\theoremstyle{definition}
\newtheorem{prop}{Proposition}
\newtheorem{thm}{Theorem}
\newtheorem{lem}{Lemma}
\newtheorem{cor}{Corollary}

\begin{document}
\maketitle

\begin{abstract}
This paper examines the simplest case of total differential equations that appears in the theory of foliation structures, without imposing the smoothness assumptions. This leads to a peculiar asymmetry in the differentiability of solutions. To resolve this asymmetry, this paper focuses on the differentiability of the integral manifold. When the system is locally Lipschitz, a solution is ensured to be only locally Lipschitz, but the integral manifolds must be $C^1$. When the system is $C^k$, we can only ensure the existence of a $C^k$ solution, but the integral manifolds must be $C^{k+1}$. In addition, we see a counterexample in which the system is $C^1$, but there is no $C^2$ solution. Moreover, we characterize a minimizer of an optimization problem whose objective function is a quasi-convex solution to a total differential equation. In this connection, we examine two necessary and sufficient conditions for the system in which any solution is quasi-convex.

\vspace{12pt}
\noindent
{\bf Keywords}: Total Differential Equation, Frobenius' Theorem, Lipschitz Functions, Quasi-Convex Functions.

\vspace{12pt}
\noindent
{\bf MSC2020 codes}: 57R30, 90C25.
\end{abstract}

\section{Introduction}
Consider the following {\bf total differential equation}:
\begin{equation}\label{TDE}
\nabla u(x)=\lambda(x)g(x),
\end{equation}
where $U\subset \mathbb{R}^n$ is open and $g:U\to \mathbb{R}^n\setminus \{0\}$ is given. It is well-known that if $g$ is $C^{\infty}$, then the following two conditions are equivalent:
\begin{enumerate}[(1)]
\item For every $x^*\in U$, there exists an open neighborhood $V\subset U$ of $x^*$, a $C^{\infty}$ nondegenerate function $u:V\to \mathbb{R}$, and a $C^{\infty}$ non-zero function $\lambda:V\to \mathbb{R}$ such that (\ref{TDE}) holds for each $x\in V$.\footnote{In this case, we say that $(u,\lambda)$ is a {\bf local solution} to (\ref{TDE}) around $x^*$.}

\item The following {\bf Jacobi's integrability condition} holds for any $x^*\in U$ and $i,j,k\in \{1,...,n\}$:\footnote{For simplicity of notation, variables are abbreviated.}
\begin{equation}\label{J}
g_i\left(\frac{\partial g_j}{\partial x_k}-\frac{\partial g_k}{\partial x_j}\right)+g_j\left(\frac{\partial g_k}{\partial x_i}-\frac{\partial g_i}{\partial x_k}\right)+g_k\left(\frac{\partial g_i}{\partial x_j}-\frac{\partial g_j}{\partial x_i}\right)=0.
\end{equation}
\end{enumerate}
This result is called {\bf Frobenius' theorem}.

The problem treated in this paper arises when we drop the $C^{\infty}$ assumption of $g$. Consider the case in which $u:U\to \mathbb{R}$ is $C^2$ and nondegenerate. Define $\lambda\equiv 1$ and $g=\nabla u$. Then, $(u,\lambda)$ is a local solution to (\ref{TDE}), where $g$ is $C^1$. Conversely, suppose that $g$ is a $C^1$ non-zero vector field on $U\subset \mathbb{R}^n$ that satisfies (\ref{J}) at any point in $U$. Frobenius' theorem can be extended to treat this case, and thus for each $x^*\in U$, there exists a local solution $(u,\lambda)$ to (\ref{TDE}) defined on some open neighborhood $V$ of $x^*$. However, this $u$ is not necessarily $C^2$. Debreu \cite{D76} constructed the following counterexample.
\[g(x_1,x_2)=\begin{cases}
\left(\frac{x_2^2}{\sqrt{1+x_2^4}},\frac{1}{\sqrt{1+x_2^4}}\right) & \mbox{if } x_2\ge 0,\\
(0,1) & \mbox{if }x_2<0.
\end{cases}\]
A solution $(u,\lambda)$ to (\ref{TDE}) around $(0,0)$ can be found, but this $u$ is $C^1$. We argue in the appendix that there is no solution $(v,\mu)$ to (\ref{TDE}) around $(0,0)$ such that $v$ is $C^2$.

This paper aims to solve this asymmetric result. Even in the above example, the integral curve must be $C^2$.\footnote{A level set of a solution to (\ref{TDE}) is sometimes called an {\bf integral manifold}.} We focus on this feature, and show that, 1) for a locally Lipschitz function $g:U\to \mathbb{R}^n\setminus \{0\}$, it satisfies (\ref{J}) almost everywhere if and only if, for any $x^*\in U$, there exists a solution $(u,\lambda)$ to (\ref{TDE}) around $x^*$ such that $u$ is locally Lipschitz and any level set of $u$ is an $n-1$ dimensional $C^1$ manifold, and 2) if $g$ is $C^k$, then $g$ satisfies (\ref{J}) everywhere if and only if, for any $x^*\in U$, there is a solution $(u,\lambda)$ around $x^*$ such that $u$ is $C^k$ and any level set of $u$ is an $n-1$-dimensional $C^{k+1}$ manifold.

We additionally address an application of this result. That is, we characterize the condition for a point to be a minimizer of a solution $u$ to (\ref{TDE}) with convex restrictions, where $u$ is {\bf quasi-convex}. If $u$ is convex, then it is well-known that the Karush--Kuhn--Tucker condition is necessary and sufficient for a solution to such a minimization problem. However, there exists $g$ such that any solution to (\ref{TDE}) is quasi-convex, but there is no convex solution to (\ref{TDE}).\footnote{See the appendix.} Our result is effective even in such a case. Together with this result, we provide two necessary and sufficient conditions for $g$ that ensure that any solution $u$ to (\ref{TDE}) is quasi-convex.

This paper is organized as follows. In Section 2, we present several mathematical notions that are heavily used in this paper. In Section 3, the main result is presented. In Section 4, the KKT property for minimization problems with quasi-convex solutions to (\ref{TDE}) is treated.

\section{Preliminary}
\subsection{General Symbols}
Before arguing for the result, we define several notions related to (\ref{TDE}). Throughout this paper, we use the following notation: for $x\in \mathbb{R}^N$ and $r>0$, $B_r(x)=\{y\in \mathbb{R}^N|\|y-x\|<r\}$ and $\bar{B}_r=\{y\in \mathbb{R}^N|\|y-x\|\le r\}$. The symbol $e_i$ denotes the $i$-th unit vector. The notation $\mathbb{R}^N_+,\mathbb{R}^N_{++}$ means the following set:
\[\mathbb{R}^N_+=\{x\in \mathbb{R}^n|x_i\ge 0\mbox{ for all }i\in \{1,...,N\}\},\]
\[\mathbb{R}^N_{++}=\{x\in \mathbb{R}^n|x_i>0\mbox{ for all }i\in \{1,...,N\}\}.\]
If $N=1$, we abbreviate this symbol, and simply write $\mathbb{R}_+,\mathbb{R}_{++}$. Moreover, for two vectors $x,y\in \mathbb{R}^N$, we write $x\ge y$ if $x-y\in\mathbb{R}^N_+$, $x\gneq y$ if $x-y\in \mathbb{R}^N_+\setminus\{0\}$, and $x\gg y$ if $x-y\in \mathbb{R}^N_{++}$. A real-valued function $f$ is said to be {\bf increasing} (resp. {\bf decreasing}) if $f(x)>f(y)$ (resp. $f(x)<f(y)$) for any $x,y$ with $x\gg y$, and {\bf strictly increasing} (resp. {\bf strictly decreasing}) if $f(x)>f(y)$ (resp. $f(x)<f(y)$) for any $x,y$ with $x\gneq y$.

\subsection{Lipschitz Functions}
Recall the Lipschitz assumption of functions. Consider a function $f:U\to \mathbb{R}^m$, where $U\subset \mathbb{R}^n$. We say that $f$ is {\bf $L$-Lipschitz} if the following inequality
\[\|f(x)-f(y)\|\le L\|x-y\|\]
holds for all $x,y\in U$. If $f$ is $L$-Lipschitz for some $L>0$, then $f$ is said to be {\bf Lipschitz}. If a restriction of $f$ into $C\subset U$ is Lipschitz, then we say that $f$ is {\bf Lipschitz on $C$}. When $U$ is open, we say that $f$ is {\bf locally Lipschitz} if $f$ is Lipschitz on any compact subset of $U$.

It is well-known that any locally Lipschitz function defined on an open set is differentiable almost everywhere. This result is called {\bf Rademacher's theorem}.

\subsection{Differentiable Manifolds}
Recall the definition of differentiable manifolds. We use only regular submanifolds of some Euclidean space, and thus we can use the definition of manifold that appears in Milnor \cite{M} or Guillemin and Pollack \cite{GP}.\footnote{The proofs of facts treated in this subsection are found in Chapter 1 of Guillemin and Pollack \cite{GP}.} First, suppose that $X\subset \mathbb{R}^N$ and $f:X\to \mathbb{R}^M$. We say that $f$ is $C^k$ if, for any $x\in X$, there exists $r>0$ and a $C^k$ mapping $F:B_r(x)\to \mathbb{R}^M$ such that if $y\in U\cap B_r(x)$, then $F(y)=f(y)$. A bijection $f:X\to Y$ is called a {\bf $C^k$ diffeomorphism} if $f$ and $f^{-1}$ are $C^k$, and two sets $X,Y$ are said to be {\bf $C^k$ diffeomorphic} if there is a $C^k$ diffeomorphism $f:X\to Y$. A set $X\subset \mathbb{R}^N$ is called an {\bf $n$-dimensional $C^k$ manifold} if, for any $x\in X$, there exists an open neighborhood $V\subset X$ of $x$ that is $C^k$ diffeomorphic to some open set $U$ in $\mathbb{R}^n$. The corresponding $C^k$ diffeomorphism $\phi:U\to V$ is called a {\bf local parametrization} around $x$, and the inverse $\phi^{-1}$ is called a {\bf local coordinate}. The dimension $n$ of $X$ is denoted by $\dim X$.

Suppose that $X\subset \mathbb{R}^N$ is an $n$-dimensional $C^k$ manifold, $x\in X$, $\phi:U\to V$ is a local parametrization around $x$, and $\phi(u)=x$. Let $D\phi(u)$ be the Jacobian matrix of $\phi$ at $u$. Define $T_x(X)$ as the range of $D\phi(u)$. It is well-known that this linear space is $n$-dimensional and independent of the choice of local parametrization $\phi$, and $v\in T_x(X)$ if and only if there exist an $\varepsilon>0$ and a $C^1$ function $c:(-\varepsilon,\varepsilon)\to X$ such that $c(0)=x$ and $v=\dot{c}(0)$. $T_x(X)$ is called the {\bf tangent space} of $X$ at $x$. A function $g:X\to \mathbb{R}^N$ is called a {\bf normal vector field} if $g(x)\cdot v=0$ for all $x\in X$ and $v\in T_x(X)$.

Let $X$ and $Y$ be $C^1$ manifolds, and $f:X\to Y$ be $C^1$. For each $x\in X$, choose $v\in T_x(X)$ and a $C^1$ curve $c(t)$ such that $c(0)=x$ and $\dot{c}(0)=v$. Then, $d\equiv f\circ c$ is a curve in $Y$ such that $d(0)=f(x)$. We define $df_x(v)=\dot{d}(0)$. It is known that $df_x(v)$ is independent of the choice of $c(t)$, and the mapping $df_x$ is a linear mapping from $T_x(X)$ into $T_{f(x)}(Y)$. We call $df_x$ the {\bf derivative} of $f$ at $x$. The usual chain rule $d(f\circ g)_x=df_{g(x)}\circ dg_x$ can easily be verified.

In particular, suppose that $\dim X=\dim Y$. It is known that $df_x$ is a bijection if and only if there exist an open neighborhood $U$ of $x$ and an open neighborhood $V$ of $f(x)$ such that $f$ is a diffeomorphism from $U$ onto $V$. This result is called the {\bf inverse function theorem}.

This theorem can be extended in two different ways. Let $X$ and $Y$ be $C^1$ manifolds. First, suppose that $\dim X\le \dim Y$, and $df_x:T_x(X)\to T_{f(x)}(Y)$ is one-to-one. Choose any local parametrization $\phi:U\to V$ of $X$ around $x$. Then, there exists a local parametrization $\psi:U'\to V'$ of $Y$ around $f(x)$ such that $(\psi^{-1}\circ f\circ \phi)(u_1,...,u_k)=(u_1,...,u_k,0,...,0)$. This result is called the {\bf local immersion theorem}. Second, suppose that $\dim X\ge \dim Y$, and $df_x:T_x(X)\to T_{f(x)}(Y)$ is onto. Choose any local parametrization $\psi:U'\to V'$ of $Y$ around $f(x)$. Then, there exists a local parametrization $\phi:U\to V$ of $X$ around $x$ such that $(\psi^{-1}\circ f\circ\phi)(u_1,...,u_k)=(u_1,...,u_{\ell})$. This result is called the {\bf local submersion theorem}.

Let $X,Y$ be $C^1$ manifolds, and $f:X\to Y$ be $C^1$. A point $y\in Y$ is called a {\bf regular value} if, for each $x\in f^{-1}(y)$, $df_x$ is onto. It is well-known that if $y$ is a regular value and $f^{-1}(y)$ is nonempty, then $f^{-1}(y)$ is a $C^1$ manifold whose dimension is $\dim X-\dim Y$. This result is called the {\bf preimage theorem}.

Finally, Let $X,Y\subset \mathbb{R}^N$ be $C^1$ manifolds and $Z=X\cap Y$. Suppose that $Z\neq \emptyset$, and for each $z\in Z$, $T_z(X)+T_z(Y)=\mathbb{R}^N$. Then, $X$ is said to be {\bf transversal} to $Y$. If so, $Z$ is a manifold whose dimension is $N-\dim X-\dim Y$.

\subsection{Ordinary Differential Equations}
Next, we mention basic knowledge on the {\bf solution function} of ordinary differential equations. Consider the following parametrized ordinary differential equation:
\begin{equation}\label{ODES}
\dot{x}(t)=h(t,x(t),z),\ x(t_0)=x_0,
\end{equation}
where $U\subset \mathbb{R}\times \mathbb{R}^N\times \mathbb{R}^M$ is open and $h:U\to \mathbb{R}^N$ is continuous. A continuously differentiable function $x(t)$ is called a {\bf solution} to (\ref{ODES}) if it is defined on an interval $I\subset \mathbb{R}$ that includes $t_0$,\footnote{We call a set $I\subset \mathbb{R}$ an {\bf interval} if $I$ is a convex subset of $\mathbb{R}$ that includes at least two different points.} $\dot{x}(t)=h(t,x(t),z)$ for all $t\in I$, and $x(t_0)=x_0$. If there are two solutions $x(t),y(t)$ to (\ref{ODES}), $x(t)$ is called an {\bf extension} of $y(t)$ if the domain of $x(t)$ includes that of $y(t)$ and $x(t)=y(t)$ if both are defined at $t$. A solution $x(t)$ to (\ref{ODES}) is called a {\bf nonextendable solution} if there is no extension of $x(t)$ except $x(t)$ itself. It is known that if $h$ is locally Lipschitz, then for all $(t_0,x_0,z)\in U$, there uniquely exists a nonextendable solution $x(t;z,x_0)$ to (\ref{ODES}). The mapping $(t;z,x_0)\mapsto x(t;z,x_0)$ is called the {\bf solution function} of (\ref{ODES}). It is also well-known that if $h$ is locally Lipschitz, then the domain of the solution function is open, and the solution function is locally Lipschitz. Moreover, if $h$ is $C^k$, then $x$ is also $C^k$.\footnote{Almost all assertions in this paragraph are proved in famous textbooks. See, for example, Pontryagin \cite{P}, Hartman \cite{H}, or Coddington and Levinson \cite{CL}. Only the local Lipschitz property of $x$ is a little hard to find a proof, but it appears in Subsection 7.1 of Hosoya \cite{H24}.}

\subsection{Convex Analysis}
Recall the definition of convex functions. A nonempty set $C\subset \mathbb{R}^n$ is said to be {\bf convex} if $(1-t)x+ty\in C$ for all $x,y\in C$ and $t\in [0,1]$. Suppose that $C$ is a convex set, and let $f:C\to \mathbb{R}$. Then, $f$ is said to be {\bf convex} if the epigraph $E=\{(x,a)|a\ge f(x)\}$ is convex. This condition is equivalent to the following: for any $x,y\in C$ and $t\in [0,1]$,
\[f((1-t)x+ty)\le (1-t)f(x)+tf(y).\]
A function $f:C\to \mathbb{R}$ is said to be {\bf concave} if $-f$ is convex. In other words, $f$ is concave if and only if, for any $x,y\in C$ and $t\in [0,1]$,
\[f((1-t)x+ty)\ge (1-t)f(x)+tf(y).\]
Sometimes, a relaxed condition for convexity is useful. Suppose that $C\subset \mathbb{R}^n$ is convex, and let $f:C\to \mathbb{R}$. Then, $f$ is said to be {\bf quasi-convex} if, for any $x\in C$, the lower contour set $L(x)=\{y\in C|f(y)\le f(x)\}$ is convex. This is equivalent to the following condition: for any $x,y\in C$ and $t\in [0,1]$,
\[f((1-t)x+ty)\le \max\{f(x),f(y)\}.\]
Because $(1-t)f(x)+tf(y)\le \max\{f(x),f(y)\}$, any convex function is quasi-convex. However, the converse is not true. For example, if $C=\mathbb{R}^2_{++}$, $f(x)=-x_1x_2$ is quasi-convex but not convex.

$f$ is said to be {\bf quasi-concave} if $-f$ is quasi-convex. In other words, $f$ is quasi-concave if and only if, for any $x,y\in C$ and $t\in [0,1]$,
\[f((1-t)x+ty)\ge \min\{f(x),f(y)\}.\]

Suppose that $U\subset \mathbb{R}^n$ is open and convex, and $f:U\to \mathbb{R}$ is convex. For each $x\in U$, define
\[\partial f(x)=\{p\in \mathbb{R}^n|\forall y\in U,\ f(y)-f(x)\ge p\cdot (y-x)\}.\]
This set-valued function $\partial f$ is called the {\bf subdifferential} of $f$. It is well-known that, if $f_1,f_2$ are two convex functions and $f=f_1+f_2$, then
\[\partial f(x)=\partial f_1(x)+\partial f_2(x).\]
See Chapter 23 of Rockafellar \cite{R} for detailed arguments.

Suppose that $C_1,C_2\subset \mathbb{R}^n$ are convex sets and $C_1\cap C_2=\emptyset$. Then, there exists $p\in \mathbb{R}^n\setminus \{0\}$ such that
\[\sup\{p\cdot x|x\in C_1\}\le \inf\{p\cdot x|x\in C_2\}.\]
This result is called the {\bf separating hyperplane theorem}. See Theorem 11.3 of Rockafellar \cite{R}.

In particular, suppose that $C\subset \mathbb{R}^n$ is a closed convex set and $x^*$ is in the boundary of $C$. Choose a sequence $(x^k)$ such that $x^k\notin C$ and $x^k\to x^*$ as $k\to \infty$. Applying the separating hyperplane theorem for the singleton $\{x^k\}$ and $C$, there exists $p^k\neq 0$ such that $p^k\cdot x\ge p^k\cdot x^k$ for all $x\in C$. Without loss of generality, we can assume that $\|p^k\|=1$, and thus, taking a subsequence, we can assume that $p^k\to p^*\neq 0$ as $k\to \infty$. Then, $p^*\cdot x\ge p^*\cdot x^*$ for all $x\in C$. This result is called the {\bf supporting hyperplane theorem}.

\section{Results}
\subsection{Normal Solutions to the Total Differential Equation}
Recall the total differential equation (\ref{TDE}):
\[\nabla u(x)=\lambda(x)g(x).\]
We assume that $U\subset \mathbb{R}^n$ is open and $g:U\to \mathbb{R}^n\setminus \{0\}$ is locally Lipschitz. A function $u:V\to \mathbb{R}$ is called a {\bf normal solution} to (\ref{TDE}) if 1) the domain $V$ of $u$ is open and included in $U$, 2) $u$ is locally Lipschitz, 3) any level set of $u$ is an $n-1$ dimensional $C^1$ manifold, 4) $g$ is a normal vector field of any level set of $u$, 5) $u$ has no local maximum point and local minimum point, and 6) if $u$ is differentiable at $x$, then $\nabla u(x)\neq 0$. If $u$ is a normal solution and $V$ includes $x^*\in U$, then $u$ is also called a {\bf normal solution around $x^*$}.

\vspace{12pt}
\noindent
{\bf NOTE}: In the introduction, we treat a pair of functions $(u,\lambda)$ as a solution to (\ref{TDE}). In contrast, in this subsection, we define a solution as a sole function $u$, and the function $\lambda$ is absent. However, the difference in these definitions is merely formal.
Actually, we can show the following result.

\begin{prop}\label{prop1}
Suppose that $U\subset \mathbb{R}^n$ is open and $g:U\to \mathbb{R}^n\setminus \{0\}$ is locally Lipschitz. If $V\subset U$ is open and $u:V\to \mathbb{R}, \lambda:V\to \mathbb{R}$ are a pair of a $C^1$ nondegenerate function and a continuous non-zero function that satisfies the relationship \textup{(\ref{TDE})} everywhere, then $u$ is a normal solution. Conversely, if $u:V\to \mathbb{R}$ is a $C^1$ normal solution, then there exists a continuous non-zero function $\lambda:V\to \mathbb{R}$ such that the relationship \textup{(\ref{TDE})} holds everywhere.
\end{prop}

\begin{proof}
Suppose that $V\subset U$ is open and a pair of functions $u:V\to \mathbb{R},\lambda:V\to \mathbb{R}$ satisfy all requirements in this theorem, and choose any $x\in V$. Define $X=u^{-1}(u(x))$. Because $u$ is nondegenerate, $u(x)$ is a regular value of $u$. By the preimage theorem, $X$ is an $n-1$ dimensional $C^1$ manifold. Let $v\in T_x(X)$, and choose a $C^1$ function $c:(-\varepsilon,\varepsilon)\to X$ such that $c(0)=x$ and $\dot{c}(0)=v$. Then, $u(c(t))\equiv u(x)$, and thus
\[\nabla u(x)\cdot v=\left.\frac{d}{dt}u(c(t))\right|_{t=0}=0.\]
Hence, $\nabla u(x)$ is orthogonal to $T_x(X)$, and thus $g(x)$ is also orthogonal to $T_x(X)$. Therefore, $g(x)$ is a normal vector field of $X$, as desired. It is clear that $u$ has no extremum point, and thus it is a normal solution to (\ref{TDE}).

Conversely, suppose that $u:V\to \mathbb{R}$ is a $C^1$ normal solution to (\ref{TDE}). Then, $\nabla u(x)\neq 0$ everywhere. As in the last paragraph, we can show that for $X=u^{-1}(u(x))$, $\nabla u(x)$ is orthogonal to $T_x(X)$. Because both $g(x)$ and $\nabla u(x)$ are orthogonal to the $n-1$ dimensional space $T_x(X)$, $\nabla u(x)$ is proportional to $g(x)$, and thus there exists $\lambda(x)\in \mathbb{R}$ such that $\nabla u(x)=\lambda(x)g(x)$. Choose any $i$ such that $g_i(x)\neq 0$ and $\frac{\partial u}{\partial x_i}(x)\neq 0$. Then,
\[\lambda(x)=\frac{\frac{\partial u}{\partial x_i}(x)}{g_i(x)},\]
where the right-hand side is continuous and non-zero around $x$. Therefore, $\lambda$ is continuous and non-zero, as desired. This completes the proof. $\blacksquare$
\end{proof}

\subsection{Key Lemmas}
In this subsection, we treat several important lemmas.

\begin{lem}\label{LEMM}
Suppose that $U\subset \mathbb{R}^n$ is open and $g:U\to \mathbb{R}^n\setminus \{0\}$ is locally Lipschitz. Moreover, suppose that $u:V\to \mathbb{R}$ is a normal solution to \textup{(\ref{TDE})}\textup{(\ref{TDE})}. Choose any $x\in V$ and $v\in \mathbb{R}^n$ such that $g(x)\cdot v\neq 0$, and define $c_x(t)=u(x+tv)$. Then, the following results hold.
\begin{enumerate}[\upshape (1)]
\item If $I$ is a compact interval such that $0\in I$, $c_x(t)$ is defined on $I$, and $g(x+tv)\cdot v\neq 0$ for all $t\in I$, then $c_x(t)$ is monotone on $I$.\footnote{In this paper, a real-valued function $f$ defined on an interval is said to be {\bf monotone} if it is either increasing or decreasing.}

\item In addition to the requirements in \textup{(1)}, suppose that $c_x(t)$ is increasing (resp. decreasing) on $I$. Then, there exists an open neighborhood $W$ of $x$ such that $c_y(t)$ is defined and increasing (resp. decreasing) on $I$ for all $y\in W$.

\item In addition to the requirements in \textup{(1)}, suppose that $0$ is not an endpoint of $I$, and $c_x(t)$ is increasing (resp. decreasing). If $g(x)\cdot w\neq 0$, then the sign of $g(x)\cdot w$ is the same as that of $g(x)\cdot v$ if and only if $t\mapsto u(x+tw)$ is increasing (resp. decreasing) around $0$.

\item If $V$ is connected and $g(x)\cdot v\neq 0$ for all $x\in V$, then $c_x(t)$ is increasing around $0$ for some $x\in V$ if and only if $c_x(t)$ is increasing around $0$ for all $x\in V$.

\item If $V$ is connected and there exists $x\in V$ such that $t\mapsto u(x+tg(x))$ is increasing around $0$, then for any $y\in V$, $t\mapsto u(y+tg(y))$ is also increasing around $0$.
\end{enumerate}
\end{lem}

\begin{proof}
Without loss of generality, we assume that $g(x)\cdot v>0$. Because $g(x+tv)\cdot v$ is continuous in $t$, $g(x+tv)\cdot v>0$ for all $t\in I$. Define $x(t)=x+tv$, and suppose that there exists $t_1,t_2\in I$ such that $t_1<t_2$ and $c_x(t_1)=c_x(t_2)$. Then, there exists $t^*\in (t_1,t_2)$ such that $t^*$ is either a local minimum point or a local maximum point of $c_x(t)$. Let $x^*=x(t^*)$, and define $X=u^{-1}(u(x^*))$. Let $i_X:X\to \mathbb{R}^n$ be the inclusion mapping: that is, $i_X(z)=z$ for all $z\in X$. Applying the local immersion theorem to this inclusion mapping, there exist an open neighborhood $W$ of $0$, an open neighborhood $Z$ of $x^*$, and a $C^1$ diffeomorphism $\varphi:W\to Z$ such that $\varphi(0)=x^*$ and for all $w\in W$, $u(\varphi(w))=u(x^*)$ if and only if $w_n=0$. Without loss of generality, we can assume that $W$ is convex.

Let $\chi=\varphi^{-1}$. Choose any $z\in T_{x^*}(X)$ and a $C^1$ curve $c(s)$ such that $c(0)=x^*$ and $\dot{c}(0)=z$. Then, $\chi_n(c(s))\equiv 0$, and thus $D\chi_n(x^*)z=0$. Hence, $\nabla \chi_n(x^*)$ is orthogonal to $T_{x^*}(X)$. Because $D\chi(x^*)$ is regular, $\nabla \chi_n(x^*)$ is non-zero, and thus there exists $a\in \mathbb{R}$ such that $a\neq 0$ and $\nabla \chi_n(x^*)=ag(x^*)$. Without loss of generality, we assume that $a>0$. Then, $\nabla \chi_n(x^*)\cdot v>0$, and thus, there exists $\delta>0$ such that $\chi_n(x(t))$ is increasing on $[t^*-\delta,t^*+\delta]$. Therefore, there exist $t_3,t_4\in [t^*-\delta,t^*+\delta]$ such that $\chi_n(t_3)<0$ and $\chi_n(t_4)>0$.

If $t^*$ is a local maximum point of $c_x(t)$, then $u(x(t_3))<u(x^*)$ and $u(x(t_4))<u(x^*)$. Suppose that there exists $z\in Z$ such that $u(z)>u(x^*)$. Then, define $f_i(s)=\varphi((1-s)\chi(z)+s\chi(x(t_i)))$. Because $W$ is convex, $f_i(s)$ is well-defined for $i\in \{3,4\}$. If $\chi_n(z)<0$, then $f_3(s)$ is a continuous function and $\chi_n(f_3(s))<0$ on $[0,1]$, and thus $u(f_3(s))\neq u(x^*)$ for all $s\in [0,1]$, which contradicts the intermediate value theorem. If $\chi_n(z)>0$, then $f_4(s)$ is a continuous function and $\chi_n(f_4(s))>0$ on $[0,1]$, which also leads a contradiction by the same arguments. Therefore, $x^*$ is a local maximum point of $u$, which is a contradiction. Symmetrically, if $t^*$ is a local minimum point of $c_x(t)$, then we can show that $x^*$ is a local minimum point of $u$, which is a contradiction. Therefore, such $t_1,t_2$ do not exist, and thus $c_x(t)$ must be monotone on the interval $I$, and (1) holds.

Next, suppose that $c_x(t)$ is increasing on $I$. Because $I$ is compact, there exists an open neighborhood $W'$ of $x$ such that $c_y(t)$ is defined on $I$ and $g(y)\cdot v\neq 0$ for all $y\in W'$. Suppose that there exists a sequence $(y^k)$ on $W'$ such that $y^k\to x$ and for all $k$, $c_{y^k}(t)$ is decreasing on $I$. Then, $c_x(t)=\lim_{k\to \infty}c_{y^k}(t)$ is nonincreasing, which is a contradiction. Therefore, there exists an open neighborhood $W$ of $x$ such that $c_y(t)$ is defined on $I$ and increasing. Symmetrically, we can show that if $c_x(t)$ is decreasing on $I$, then there exists an open neighborhood $W$ of $x$ such that $c_y(t)$ is defined on $I$ and decreasing for any $y\in W$. Hence, (2) holds.

Third, suppose that $0$ is not an endpoint of $I$, and $c_x(t)$ is increasing on $I$. If $g(x)\cdot w$ has the same sign as $g(x)\cdot v$, then define $v(s)=(1-s)v+sw$. Then, $g(x)\cdot v(s)\neq 0$ for all $s\in [0,1]$. Let $I_1$ be the set of $s\in [0,1]$ such that $t\mapsto u(x+tv(s))$ is increasing around $0$. Then, $0\in I_1$. It can be easily shown that $I_1$ and $[0,1]\setminus I_1$ are closed, and because $[0,1]$ is connected, $I_1=[0,1]$. Hence, $t\mapsto u(x+tw)$ is increasing around $0$. Conversely, if the sign of $g(x)\cdot w$ is different from that of $g(x)\cdot v$, define $v(s)=(1-s)(-v)+sw$. Then, by the same arguments as above, we can show that $t\mapsto u(x+tv(s))$ is decreasing around $0$ for all $s\in [0,1]$, and in particular, $t\mapsto u(x+tw)$ is decreasing around $0$. The case in which $c_x(t)$ is decreasing can be treated symmetrically, and thus the proof in such a case is omitted. Hence, (3) holds.

Fourth, suppose that $V$ is connected and $g(x)\cdot v\neq 0$ for all $x\in U$. Let $U_1$ be the set of all $x\in V$ such that $c_x(t)$ is increasing around $0$. By the above arguments, we have that both $U_1$ and $V\setminus U_1$ are open. Because $V$ is connected, either $U_1=V$ or $U_1=\emptyset$, and (4) holds.

Finally, suppose again that $V$ is connected and there exists $x\in V$ such that $t\mapsto u(x+tg(x))$ is increasing around $0$. Let $V_1$ be the set of all $y\in V$ such that $t\mapsto u(y+tg(y))$ is increasing around $0$. By our previous arguments, we have that $V_1$ and $V\setminus V_1$ are open. Because $x\in V_1$ and $V$ is connected, we have that $V=V_1$, and (5) holds. This completes the proof. $\blacksquare$
\end{proof}

Suppose that $u:V\to \mathbb{R}$ is a normal solution to (\ref{TDE}). We say that $u$ is a {\bf regular solution} to (\ref{TDE}) if $V$ is connected and for any $x\in V$, $t\mapsto u(x+tg(x))$ is increasing around $0$. Note that, by (5) of Lemma \ref{LEMM}, for any normal solution $u$ defined on a connected set, either $u$ or $-u$ is regular.

Next, we consider another problem related to (\ref{TDE}). Let $U\subset \mathbb{R}^n\times \mathbb{R}$ be open and convex, and a function $f:U\to \mathbb{R}^n$ be locally Lipschitz. If $f$ is differentiable at $(x,y)$, define
\[s_{ij}(x,y)=\frac{\partial f_i}{\partial x_j}(x,y)+\frac{\partial f_i}{\partial y}(x,y)f_j(x,y).\]
We say that $f$ is {\bf integrable} if $s_{ij}(x,y)=s_{ji}(x,y)$ for almost all $(x,y)\in U$. Note that, by Rademacher's theorem, $f$ is differentiable almost everywhere.

The following lemma is quite important in this article.

\begin{lem}\label{KEY}
Suppose that $U\subset \mathbb{R}^n\times \mathbb{R}$ is open and $f:U\to \mathbb{R}^n$ is locally Lipschitz. Suppose that $(x^*,y^*)\in U$ and $V$ is an open and star-convex set centered at $x^*$.\footnote{Recall that $V$ is called a star-convex set centered at $x^*$ if for any $x\in V$ and $t\in [0,1]$, $(1-t)x^*+tx\in V$.} Consider the following partial differential equation:
\begin{equation}\label{PDE}
\nabla u(x)=f(x,u(x)),\ u(x^*)=y^*,
\end{equation}
and the following ordinary differential equation:
\begin{equation}\label{ODE}
\dot{c}(t)=f((1-t)x^*+tx,c(t))\cdot (x-x^*),\ c(0)=y^*.
\end{equation}
If $f$ is integrable, then there exists a solution to \textup{(\ref{PDE})} on $V$ if and only if the domain of the solution function $c(t;x)$ of \textup{(\ref{ODE})} includes $[0,1]\times V$. Moreover, in this case, $u(x)\equiv c(1;x)$ is a unique solution to \textup{(\ref{PDE})}.
\end{lem}

In Theorem 10.9.4 of Dieudonne \cite{D}, the partial differential equation (\ref{PDE}) was considered, where $U$ is a subset of two Banach spaces $X,Y$, and the value of $f$ is a bounded linear operator from $X$ into $Y$. Hence, our (\ref{PDE}) is a minor version of Dieudonne's one. However, Dieudonne assumed that $f$ is $C^1$. In contrast, we only assume that $f$ is locally Lipschitz. This extension was proved in Hosoya \cite{H21, H21B}. Hence, we omit the proof. Note that, if $f$ is $C^k$, then by the bootstrap argument, we can show that the solution $u$ to (\ref{PDE}) is $C^{k+1}$.

Next, recall Jacobi's integrability condition (\ref{J}):
\begin{equation}\label{J2}
g_i\left(\frac{\partial g_j}{\partial x_k}-\frac{\partial g_k}{\partial x_j}\right)+g_j\left(\frac{\partial g_k}{\partial x_i}-\frac{\partial g_i}{\partial x_k}\right)+g_k\left(\frac{\partial g_i}{\partial x_j}-\frac{\partial g_j}{\partial x_i}\right)=0.
\end{equation}
We use the following lemmas.

\begin{lem}\label{KEY2}
Suppose that $U\subset \mathbb{R}^n$ is open and $g:U\to \mathbb{R}^n\setminus\{0\}$ is locally Lipschitz. Choose any $x\in U$ such that $g$ is differentiable at $x$ and $g_{k^*}(x)\neq 0$. Then, the following conditions are equivalent.
\begin{enumerate}[\upshape i)]
\item $g$ satisfies \textup{(\ref{J2})} at $x$ for any $i,j,k\in \{1,...,n\}$.

\item $g$ satisfies \textup{(\ref{J2})} at $x$ for any $i,j,k\in \{1,...,n\}$ with $i\neq j\neq k\neq i$.

\item $g$ satisfies \textup{(\ref{J2})} at $x$ for any $i,j,k\in \{1,...,n\}$ with $i\neq j\neq k\neq i$ and $k=k^*$.
\end{enumerate}
\end{lem}

\begin{proof}
It is clear that i) implies ii) and ii) implies iii). Choose any $i,j,k\in \{1,...,n\}$. If $i=j$, then
\begin{align*}
&~g_i\left(\frac{\partial g_j}{\partial x_k}-\frac{\partial g_k}{\partial x_j}\right)+g_j\left(\frac{\partial g_k}{\partial x_i}-\frac{\partial g_i}{\partial x_k}\right)+g_k\left(\frac{\partial g_i}{\partial x_j}-\frac{\partial g_j}{\partial x_i}\right)\\
=&~g_i\left(\frac{\partial g_i}{\partial x_k}-\frac{\partial g_k}{\partial x_i}\right)+g_i\left(\frac{\partial g_k}{\partial x_i}-\frac{\partial g_i}{\partial x_k}\right)=0,
\end{align*}
and thus (\ref{J2}) holds. By the same reason, (\ref{J2}) must hold if $j=k$ or $i=k$. Therefore, ii) implies i).

It suffices to show that iii) implies ii). Suppose that iii) holds, and let $i\neq j\neq k\neq i$. If $k^*\in \{i,j,k\}$, then (\ref{J2}) holds by iii). Therefore, we assume that $k^*\notin\{i,j,k\}$. By iii),
\begin{align*}
&~g_i\left(\frac{\partial g_j}{\partial x_{k^*}}-\frac{\partial g_{k^*}}{\partial x_j}\right)+g_j\left(\frac{\partial g_{k^*}}{\partial x_i}-\frac{\partial g_i}{\partial x_{k^*}}\right)+g_{k^*}\left(\frac{\partial g_i}{\partial x_j}-\frac{\partial g_j}{\partial x_i}\right)=0,\\
&~g_j\left(\frac{\partial g_k}{\partial x_{k^*}}-\frac{\partial g_{k^*}}{\partial x_k}\right)+g_k\left(\frac{\partial g_{k^*}}{\partial x_j}-\frac{\partial g_j}{\partial x_{k^*}}\right)+g_{k^*}\left(\frac{\partial g_j}{\partial x_k}-\frac{\partial g_k}{\partial x_j}\right)=0,\\
&~g_k\left(\frac{\partial g_i}{\partial x_{k^*}}-\frac{\partial g_{k^*}}{\partial x_i}\right)+g_i\left(\frac{\partial g_{k^*}}{\partial x_k}-\frac{\partial g_k}{\partial x_{k^*}}\right)+g_{k^*}\left(\frac{\partial g_k}{\partial x_i}-\frac{\partial g_i}{\partial x_k}\right)=0.
\end{align*}
Multiplying $g_k$ to the first line, $g_i$ to the second line, $g_j$ to the third line, and summing them up, we obtain that
\[g_{k^*}\left[g_i\left(\frac{\partial g_j}{\partial x_k}-\frac{\partial g_k}{\partial x_j}\right)+g_j\left(\frac{\partial g_k}{\partial x_i}-\frac{\partial g_i}{\partial x_k}\right)+g_k\left(\frac{\partial g_i}{\partial x_j}-\frac{\partial g_j}{\partial x_i}\right)\right]=0,\]
and because $g_{k^*}(x)\neq 0$, (\ref{J2}) holds for $i,j,k$. This completes the proof. $\blacksquare$
\end{proof}

We say that $g$ satisfies {\bf Jacobi's integrability condition} if, for any $i,j,k\in \{1,...,n\}$, (\ref{J2}) holds almost everywhere.

\begin{lem}\label{KEY3}
Suppose that $U\subset \mathbb{R}^n$ is open and $g:U\to \mathbb{R}^n\setminus\{0\}$ is locally Lipschitz. Choose $x\in U$ such that $g$ is differentiable at $x$ and $g_n(x)\neq 0$, and for $i\in \{1,...,n-1\}$, define
\[f_i(y)=-\frac{g_i(y)}{g_n(y)}.\]
Then, $g$ satisfies \textup{(\ref{J2})} for any $i,j,k\in \{1,...,n\}$ at $x$ if and only if $s_{ij}(x)=s_{ji}(x)$ for any $i,j\in \{1,...,n-1\}$.
\end{lem}

\begin{proof}
Choose any $i,j\in \{1,...,n-1\}$. By direct calculation,
\begin{align*}
s_{ij}=&~\frac{\partial f_i}{\partial x_j}+\frac{\partial f_i}{\partial x_n}f_j\\
=&~-\frac{\frac{\partial g_i}{\partial x_j}g_n-g_i\frac{\partial g_n}{\partial x_j}}{g_n^2}+\frac{\frac{\partial g_i}{\partial x_n}g_n-g_i\frac{\partial g_n}{\partial x_n}}{g_n^2}\times\frac{g_j}{g_n}\\
=&~\frac{1}{g_n^2}\left[-g_n\frac{\partial g_i}{\partial x_j}+g_i\frac{\partial g_n}{\partial x_j}+g_j\frac{\partial g_i}{\partial x_n}\right]-\frac{g_ig_j\frac{\partial g_n}{\partial x_n}}{g_n^3}.
\end{align*}
Therefore,
\begin{align*}
&~s_{ij}=s_{ji}\Leftrightarrow s_{ji}-s_{ij}=0\\
\Leftrightarrow&~\frac{1}{g_n^2}\left[g_n\left(\frac{\partial g_i}{\partial x_j}-\frac{\partial g_j}{\partial x_i}\right)+g_i\left(\frac{\partial g_j}{\partial x_n}-\frac{\partial g_n}{\partial x_j}\right)+g_j\left(\frac{\partial g_n}{\partial x_i}-\frac{\partial g_i}{\partial x_n}\right)\right]=0.
\end{align*}
Hence, by Lemma \ref{KEY2}, our claim is correct. This completes the proof. $\blacksquare$
\end{proof}

By Lemma \ref{KEY3}, if $g_n(x)\neq 0$ for any $x\in U$, then $g$ satisfies Jacobi's integrability condition if and only if $f=-\frac{g_i}{g_n}$ is integrable. This result is quite important for our main result.

The next lemma is needed to prove the next proposition.\footnote{In the rest of this section, we use the following notation: if $x\in \mathbb{R}^n$, then $\tilde{x}$ denotes $(x_1,...,x_{n-1})\in \mathbb{R}^{n-1}$. In particular, $\tilde{x}^*=(x_1^*,...,x_{n-1}^*)$.}

\begin{lem}\label{EQ}
Let $U\subset \mathbb{R}^n$ be open and $g:U\to \mathbb{R}^n\setminus \{0\}$ be locally Lipschitz. Choose $x^*\in U$ such that $g_n(x^*)\neq 0$. Suppose that there exists a normal solution $u:V\to \mathbb{R}$ to \textup{(\ref{TDE})} around $x^*$ such that $u$ is increasing in $x_n$ around $x^*$, and choose any $\varepsilon>0$ such that $V'=\prod_{i=1}^n(x_i^*-\varepsilon,x_i^*+\varepsilon)\subset V$ and $g_n(x)\neq 0$ for any $x\in V'$. Choose any $\delta>0$ such that $\delta<\varepsilon$ and for any $\tilde{x}\in \prod_{i=1}^{n-1}(x_i^*-\delta,x_i^*+\delta)$ and $x_n\in (x_n^*-\delta,x_n^*+\delta)$,
\[u(\tilde{x},x_n^*-\varepsilon)<u(\tilde{x}^*,x_n)<u(\tilde{x},x_n^*+\varepsilon)\]
Define $W=\prod_{i=1}^{n-1}(x_i^*-\delta,x_i^*+\delta)$, $I=(x_n^*-\delta,x_n^*+\delta)$, and $W'=W\times I$. Then, for each $x_n\in I$, there uniquely exists a $C^1$ function $E^{x_n}:W\to (x_n^*-\varepsilon,x_n^*+\varepsilon)$ such that $u(\tilde{x},E^{x_n}(\tilde{x}))=u(\tilde{x}^*,x_n)$ for any $x\in W'$. Moreover, if we define $f_i(x)=-\frac{g_i(x)}{g_n(x)}$, then
\begin{equation}\label{LEVV}
\frac{\partial E^{x_n}}{\partial x_i}(\tilde{x})=f_i(\tilde{x},E^{x_n}(\tilde{x})),\ E^{x_n}(\tilde{x}^*)=x_n.
\end{equation}
Furthermore, if $g$ is differentiable at $x^*$, then $g$ satisfies \textup{(\ref{J2})} for any $i,j,k\in \{1,...,n\}$ at $x^*$.
\end{lem}

\begin{proof}
By Lemma \ref{LEMM}, we have that $u$ is increasing in $x_n$ on $V'$. Therefore, the `uniqueness' claim of this lemma is obvious.

Choose any $x\in W'$. By the intermediate value theorem, there exists $E^{x_n}(\tilde{x})\in (x_n^*-\varepsilon,x_n^*+\varepsilon)$ such that $u(\tilde{x},E^{x_n}(\tilde{x}))=u(\tilde{x}^*,x_n)$. Define $z=(\tilde{x},E^{x_n}(\tilde{x}))$ and $X=u^{-1}(u(z))$. By our assumption, $X$ is an $n-1$ dimensional $C^1$ manifold. For each $y\in X\cap V'$, define $\chi(y)=\tilde{y}$. Then, $\chi^{-1}(w)=(w,E^{x_n}(w))$ around $\tilde{x}$. Because $\chi$ is a restriction of the linear mapping $y\mapsto \tilde{y}$ from $\mathbb{R}^n$ into $\mathbb{R}^{n-1}$, $\chi$ is $C^{\infty}$. Since $g_n(z)\neq 0$ and $g(z)$ is tangent to $T_z(X)$, $e_n\notin T_z(X)$. Choose a basis $v_1,...,v_{n-1}$ of $T_z(X)$. Suppose that $a_1,...,a_{n-1}\in \mathbb{R}$ and
\[a_1\tilde{v}_1+...+a_{n-1}\tilde{v}_{n-1}=0.\]
Then, there exists $a_n\in \mathbb{R}$ such that
\[a_1v_1+...+a_{n-1}v_{n-1}=a_ne_n.\]
Because $e_n\notin T_z(X)$, $a_n=0$, and thus $a_1=...=a_{n-1}=0$. Therefore, $\tilde{v}_1,...,\tilde{v}_{n-1}$ are linearly independent, and $d\chi_z:T_z(X)\to \mathbb{R}^{n-1}$ is a bijection. By the inverse function theorem, we have that $\chi^{-1}$ is $C^1$ around $\tilde{x}$, and in particular, $E^{x_n}$ is $C^1$ around $\tilde{x}$.

Choose any $\tilde{x}\in W$ and $x_n\in I$, and let $z=(\tilde{x},E^{x_n}(\tilde{x}))$. Because $c_i(t)=(\tilde{x}+t\tilde{e}_i,E^{x_n}(\tilde{x}+t\tilde{e}_i))$ is a $C^1$ curve on $X$, we have that $\dot{c}_i(0)\in T_y(X)$, and thus $\dot{c}_i(0)\cdot g(y)=0$. Hence,
\[g_i(z)+g_n(z)\frac{\partial E^{x_n}}{\partial x_i}(\tilde{x})=0,\]
and thus $E^{x_n}$ satisfies (\ref{LEVV}).

If $g$ is differentiable at $x^*$, then $E^{x_n^*}$ is twice differentiable at $x^*$. By a simple calculation, we have that
\[s_{ij}(x^*)=\frac{\partial^2E^{x_n^*}}{\partial x_j\partial x_i}(\tilde{x}^*)=\frac{\partial^2E^{x_n^*}}{\partial x_i\partial x_j}(\tilde{x}^*)=s_{ji}(x^*).\]
By Lemma \ref{KEY3}, for each $i,j,k\in \{1,...,n\}$, (\ref{J2}) holds at $x^*$. This completes the proof. $\blacksquare$
\end{proof}

\subsection{Invariance of the Level Sets}
Using our Lemmas, we can show the following proposition.

\begin{prop}\label{prop2}
Suppose that $U\subset \mathbb{R}^n$ is open and $g:U\to \mathbb{R}^n\setminus\{0\}$ is locally Lipschitz. If $u_1:V_1\to \mathbb{R}$ and $u_2:V_2\to \mathbb{R}$ are two normal solutions to \textup{(\ref{TDE})} around $x^*$, then there exists an open neighborhood $V_3\subset V_1\cap V_2$ of $x^*$ such that $u_2$ is a monotone (increasing or decreasing) transform of $u_1$ on $V_3$. In particular, $u_1^{-1}(u_1(x^*))\cap V_3=u_2^{-1}(u_2(x^*))\cap V_3$.
\end{prop}

\begin{proof}
Because $g(x^*)\neq 0$, we can assume without loss of generality that $g_n(x^*)\neq 0$. First, choose $\varepsilon>0$ such that $V^*\equiv \prod_{i=1}^n[x_i^*-\varepsilon,x_i^*+\varepsilon]\subset V_1\cap V_2$ and $g_n(x)\neq 0$ for any $x\in V^*$. Replacing $u_i$ with $-u_i$ if necessary, we can assume without loss of generality that $u_i$ is increasing in $x_n$ on $V^*$.

By Lemma \ref{EQ}, $g$ satisfies Jacobi's integrability condition in the interior of $V^*$. Define $f_i=-\frac{g_i}{g_n}$, and consider the following parametrized differential equation:
\begin{equation}\label{LEV}
\dot{z}(t)=f((1-t)\tilde{x}+t\tilde{x}^*,z(t))\cdot (\tilde{x}^*-\tilde{x}),\ z(0)=x_n.
\end{equation}
If $x=x^*$, then $z(t)\equiv x_n^*$ is a solution defined on $\mathbb{R}$. Thus, there exists $\delta>0$ such that $\delta<\varepsilon$, the solution function $z(t;x)$ of (\ref{LEV}) is defined on $[0,1]\times \prod_{j=1}^n[x_j^*-\delta,x_j^*+\delta]$, and $|z(t;x)-x_n^*|<\varepsilon$ for all $t\in [0,1]$ and $x\in \prod_{j=1}^n[x_j^*-\delta,x_j^*+\delta]$. Without loss of generality, we can assume that, for any $x\in \prod_{j=1}^n[x_j^*-\delta,x_j^*+\delta]$ and $i\in \{1,2\}$,
\[u_i(\tilde{x},x_n^*-\varepsilon)<u_i(\tilde{x}^*,x_n)<u_i(\tilde{x},x_n^*+\varepsilon).\]
Define $W=\prod_{j=1}^{n-1}(x_j^*-\delta,x_j^*+\delta)$, $I=(x_n^*-\delta,x_n^*+\delta)$, and $W'=W\times I$. Choose any $x\in W'$. By Lemma \ref{EQ}, for each $i\in \{1,2\}$, there uniquely exists $E_i^{x_n}(\tilde{x})\in (x_n^*-\varepsilon,x_n^*+\varepsilon)$ such that $u_i(\tilde{x},E_i^{x_n}(\tilde{x}))=u_i(\tilde{x}^*,x_n)$. Moreover, each $E_i^{x_n}$ is $C^1$, and it is a solution to (\ref{LEVV}). By Lemma \ref{KEY},
\[E_i^{x_n}(\tilde{x})=z(1;x),\]
where the right-hand side is independent of $i$. Hence, $E_1^{x_n}=E_2^{x_n}$ on $W$. Let
\[V_3=\prod_{j=1}^n(x_j^*-\delta,x_j^*+\delta)\cap u_1^{-1}(\{u_1(\tilde{x}^*,x_n)|x_n\in I\}).\]
Then, $V_3$ is an open neighborhood of $x^*$. For any $x\in V_3$, there uniquely exists $\psi(x)\in (x_n^*-\delta,x_n^*+\delta)$ such that
\[u_1(x)=u_1(\tilde{x}^*,\psi(x)),\]
which implies that
\[E_i^{\psi(x)}(\tilde{x})=x_n,\]
and thus,
\[u_2(x)=u_2(\tilde{x}^*,\psi(x)).\]
Let $\phi_i(c)=u_i(\tilde{x}^*,c)$. Then, for each $x\in V_3$,
\[u_2(x)=\phi_2(\psi(x))=(\phi_2\circ \phi_1^{-1})(\phi_1(\psi(x)))=(\phi_2\circ \phi_1^{-1})(u_1(x)),\]
and thus $u_2$ is a monotone transform of $u_1$ on $V_3$, as desired. This completes the proof. $\blacksquare$
\end{proof}

When some additional requirements hold, then this result can be extended.

\begin{cor}\label{COR}
Suppose that $U\subset \mathbb{R}^n$ is open and $g:U\to \mathbb{R}^n\setminus \{0\}$ is locally Lipschitz. Suppose also that $u_1$ and $u_2$ are two normal solutions to \textup{(\ref{TDE})} defined on the same connected domain $V$, and any level set of $u_i$ is connected. Then, $u_2$ is a monotone (increasing, or decreasing) transform of $u_1$ on $V$.
\end{cor}

\begin{proof}
Without loss of generality, we can assume that each $u_i$ is regular. First, we show that for any $x\in V$, $u_1^{-1}(u_1(x))=u_2^{-1}(u_2(x))$. Let $X=u_1^{-1}(u_1(x))$. Because $u_1$ is a normal solution to (\ref{TDE}), $X$ is an $n-1$ dimensional connected $C^1$ manifold. Note that, for a topological manifold, connectedness is equivalent to path connectedness.\footnote{In general, path connectedness implies connectedness. To prove the converse, choose any $y\in X$ and define $U_1$ as the set of all $z$ such that there exists a continuous function $c:[0,1]\to X$ with $c(0)=y, c(1)=z$. Then, we can easily show that $U_1$ is clopen whenever $X$ is a topological manifold. Hence, if $X$ is connected, then $U_1=X$, as desired.} Hence, $X$ is path connected. Choose any $y\in X$ and let $c:[0,1]\to X$ be a continuous function such that $c(0)=x, c(1)=y$. Define $d(t)=u_2(c(t))$. Choose any $t^*\in [0,1]$ and define $x^*=x(t^*)$. By Proposition \ref{prop2}, $d(t)$ is locally constant around $t^*$. Because $t^*$ is arbitrary, we conclude that $d$ is constant and $d(1)=d(0)$, which implies that $u_2(y)=u_2(x)$. Symmetrically, we can show that if $u_2(y)=u_2(x)$, then $u_1(y)=u_1(x)$, and thus $X=u_2^{-1}(u_2(x))$, as desired.

Let $I=u_1(V)$, and for each $c\in I$, define $\phi(c)=u_2(u_1^{-1}(c))$. Due to the above relationship, the definition of $\phi$ has no ambiguity, and $u_2=\phi\circ u_1$. Choose any $c\in I$ and suppose that $(c_k)$ is a sequence in $I$ that converges to $c$. Choose any $x\in u_1^{-1}(c)$. Then, there exists $(t_k)$ such that $t_k\to 0$ and for any sufficiently large $k$, $u_1(x+t_kg(x))=c_k$. Therefore, 
\[\phi(c_k)=u_2(x+t_kg(x))\to u_2(x)=\phi(c),\]
which implies that $\phi$ is continuous. Because of Proposition \ref{prop2}, we have that $\phi(c)$ is locally increasing. If $\phi(c)\ge \phi(d)$ for some $c,d\in I$ with $c<d$, then there exists an extremum of $\phi$ in $(c,d)$, which contradicts the above arguments. Therefore,  $\phi$ is increasing, as desired. This completes the proof. $\blacksquare$
\end{proof}

\begin{cor}\label{COR2}
Suppose that $I\subset \mathbb{R}$ is an open interval, and let $U=I^n$ and $g:U\to \mathbb{R}^n_+\setminus \{0\}$ is locally Lipschitz. Then, any regular solution $u$ to (\ref{TDE}) defined on $U$ is increasing, and any level set of $u$ is connected. In particular, if $u_1:U\to \mathbb{R}$ and $u_2:U\to \mathbb{R}$ are normal solutions to (\ref{TDE}), then $u_2$ is a monotone transform of $u_1$.
\end{cor}

\begin{proof}
Suppose that $u$ is a regular solution to (\ref{TDE}). Let $v\gg 0$. Then, for any $x\in U$, $g(x)\cdot v>0$. By Lemma \ref{LEMM}, we have that $t\mapsto u(x+tv)$ is increasing around $0$. If $y\gg x$, then define $c(t)=u((1-t)x+ty)$. Then, $c(t)$ is continuous and locally increasing on $[0,1]$, which implies that it is globally increasing. Therefore, $u(y)=c(1)>c(0)=u(x)$, and thus $u$ is increasing.

Second, choose $x,y\in U$ such that $x\neq y$ and $u(x)=u(y)$, and let $X=u^{-1}(u(x))$. Let $z_i=\min\{x_i,y_i\}$ and $w_i=\max\{x_i,y_i\}+1$. Define
\[x(t)=\begin{cases}
(1-t)x+tz & \mbox{if }0\le t\le 1,\\
(2-t)z+(t-1)y & \mbox{if }1\le t\le 2,
\end{cases}\]
and $v=w-z$. Then, $v\gg 0$, and because $u$ is increasing and continuous, for each $t\in [0,2]$, there uniquely exists $s(t)\in [0,1]$ such that $u(x(t)+s(t)v)=u(x)$. By Berge's theorem, we can show that $s(t)$ is continuous, and thus $c(t)=x(t)+s(t)v$ is a continuous function from $[0,2]$ into $X$ such that $c(0)=x$ and $c(2)=y$, which implies that $X$ is path connected, and thus it is connected.

Hence, if $u_1,u_2$ are two normal solutions to (\ref{TDE}), then they satisfy the connectedness requirement of level sets in Corollary \ref{COR}, and by this corollary, $u_2$ is a monotone transform of $u_1$. This completes the proof. $\blacksquare$
\end{proof}

\subsection{Main Result}
The main result is as follows.

\begin{thm}\label{MAIN}
Suppose that $U\subset \mathbb{R}^n$ is open and $g:U\to \mathbb{R}^n\setminus \{0\}$ is locally Lipschitz. Then, $g$ satisfies Jacobi's integrability condition if and only if, for any $x^*\in U$, there exists a normal solution $u:V\to \mathbb{R}$ to \textup{(\ref{TDE})} around $x^*$. Moreover, when $g$ is $C^k$, then it satisfies Jacobi's integrability condition if and only if, for any $x^*\in U$, there exists a $C^k$ normal solution $u:V\to \mathbb{R}$ to \textup{(\ref{TDE})} around $x^*$ such that any level set of $u$ is an $n-1$ dimensional $C^{k+1}$ manifold.
\end{thm}

\begin{proof}
Recall that, if $x\in \mathbb{R}^n$, then the symbol $\tilde{x}$ denotes $(x_1,...,x_{n-1})$.

Suppose that, for any $x^*\in U$, there exists a normal solution to (\ref{TDE}) around $x^*$. Then, by Lemma \ref{EQ}, $g$ satisfies Jacobi's integrability condition. Therefore, it suffices to show the converse. Suppose that $g$ satisfies Jacobi's integrability condition. Choose $x^*\in U$. Because $g(x^*)\neq 0$, we can assume without loss of generality that $g_n(x^*)\neq 0$. Define
\[f_i(x)=-\frac{g_i(x)}{g_n(x)}.\]
By Lemma \ref{KEY3}, $f$ is defined whenever $g_n(x)\neq 0$, and is integrable. Consider the following differential equation:
\begin{equation}\label{ODE2}
\dot{c}(t)=f((1-t)\tilde{x}^*+t\tilde{x},c(t))\cdot (\tilde{x}-\tilde{x}^*),\ c(0)=z.
\end{equation}
Let $c(t;\tilde{x},z)$ be the solution function of (\ref{ODE2}). It is clear that $c(t;\tilde{x}^*,x_n^*)\equiv x_n^*$, which is defined on $\mathbb{R}$. Because the domain $U^*$ of the solution function $c$ is open and includes $[0,1]\times \{x^*\}$, there exists $\varepsilon>0$ such that $[0,1]\times \prod_{i=1}^n[x_i^*-\varepsilon,x_i^*+\varepsilon]\subset U^*$. Define
\[C=\{c(t;\tilde{x},z)|t\in [0,1],\ |x_i-x_i^*|\le \varepsilon,\ |z-x_n^*|\le \varepsilon\}.\]
Then, $C$ is compact, and thus there exists a constant $L>0$ such that $f$ is $L$-Lipschitz on $\prod_{i=1}^{n-1}[x_i^*-\varepsilon,x_i^*+\varepsilon]\times C$ and $c$ is $L$-Lipschitz on $[0,1]\times \prod_{i=1}^n[x_i^*-\varepsilon,x_i^*+\varepsilon]$.

Because $f$ is integrable, for each $z\in [x_n^*-\varepsilon,x_n^*+\varepsilon]$, $E^z(\tilde{x})=c(1;\tilde{x},z)$ is the unique $C^1$ solution to the following differential equation:
\begin{equation}\label{PDE3}
\frac{\partial E}{\partial x_i}(\tilde{x})=f_i(\tilde{x},E(\tilde{x})),\ E(\tilde{x}^*)=z
\end{equation}
defined on $\prod_{i=1}^{n-1}(x_i^*-\varepsilon,x_i^*+\varepsilon)$. If $z>z'$ and $c(t;\tilde{x},z)\le c(t;\tilde{x},z')$ for some $t\in [0,1]$, then by the intermediate value theorem, there exists $t^*\in [0,1]$ such that $c(t^*;\tilde{x},z)=c(t^*;\tilde{x},z')$. By the uniqueness of the solution to (\ref{ODE2}), we have that $z=c(0;\tilde{x},z)=c(0;\tilde{x},z')=z'$, which is a contradiction. Therefore, $c(t;\tilde{x},z)$ is increasing in $z$, and in particular, $E^z(\tilde{x})$ is increasing in $z$. Because the mapping $c(t;\tilde{x},z)$ is continuous, $E^z(\tilde{x})$ is continuous in $(z,\tilde{x})$. Hence, there exists $\delta>0$ such that $\delta<\varepsilon$ and for any $\tilde{x}\in \prod_{i=1}^{n-1}[x_i^*-\delta,x_i^*+\delta]$, $\|\tilde{x}-\tilde{x}^*\|\le \frac{1}{2L^2}$ and
\[E^{x_n^*-\varepsilon}(\tilde{x})<x_n^*-\delta<x_n^*+\delta<E^{x_n^*+\varepsilon}(\tilde{x}).\]
By the intermediate value theorem, if $|x_i-x_i^*|\le \delta$ for all $i\in \{1,...,n\}$, then there uniquely exists $z^*\in (x_n^*-\varepsilon,x_n^*+\varepsilon)$ such that $E^{z^*}(\tilde{x})=x_n$. Define $u(x)$ as such $z^*$. Let $V=\prod_{i=1}^n(x_i^*-\delta,x_i^*+\delta)$, and $\bar{V}=\prod_{i=1}^n[x_i^*-\delta,x_i^*+\delta]$. Then, for each $x\in \bar{V}$,
\[u(x)=z^*\Leftrightarrow E^{z^*}(\tilde{x})=x_n\Leftrightarrow c(1;\tilde{x},z^*)=x_n.\]
Therefore, $\tilde{x}\mapsto (\tilde{x},E^{z^*}(\tilde{x}))$ is a $C^1$ local parametrization of $u^{-1}(z^*)\cap V$. Hence, $X\equiv u^{-1}(z^*)$ is an $n-1$ dimensional $C^1$ manifold, and for $v_i=e_i+f_i(x)e_n$, $\{v_1,...,v_{n-1}\}$ consists of the basis of $T_x(X)$. However,
\[v_i\cdot g(x)=g_i(x)+f_i(x)g_n(x)=0,\]
which implies that $g$ is a normal vector field of $X$.

Next, we show that $u$ is Lipschitz on $V$. To show this, we first show that $u$ is Lipschitz in $x_n$ on $V$. Suppose that $\tilde{x}\in \prod_{i=1}^{n-1}[x_i^*-\delta,x_i^*+\delta]$, and let $\tilde{x}(t)=(1-t)\tilde{x}^*+t\tilde{x}$. If $z,z'\in [x_n^*-\varepsilon,x_n^*+\varepsilon]$ and $z'>z$, then $(\tilde{x}(t),c(t;\tilde{x},z)),(\tilde{x}(t),c(t;\tilde{x},z'))\in \prod_{i=1}^{n-1}[x_i^*-\varepsilon,x_i^*+\varepsilon]\times C$ for any $t\in [0,1]$. Moreover, $c(t;\tilde{x},z')>c(t;\tilde{x},z)$ for all $t\in [0,1]$. Thus,
\begin{align*}
&~c(1;\tilde{x},z')-c(1;\tilde{x},z)\\
=&~z'-z+\int_0^1(f(\tilde{x}(t),c(t;\tilde{x},z'))-f(\tilde{x}(t),c(t;\tilde{x},z)))\cdot (\tilde{x}-\tilde{x}^*)dt\\
\ge&~z'-z-\int_0^1\|\tilde{x}-\tilde{x}^*\|L(c(t;\tilde{x},z')-c(t;\tilde{x},z))dt\\
\ge&~\frac{1}{2}(z'-z).
\end{align*}
Therefore, if $(\tilde{x},x_n),(\tilde{x},x_n')\in \bar{V}$ and $x_n'>x_n$, then for $z=u(\tilde{x},x_n)$ and $z'=u(\tilde{x},x_n')$, $z'>z$ and
\[x_n'-x_n=c(1;\tilde{x},z')-c(1;\tilde{x},z)\ge \frac{z'-z}{2},\]
which implies that
\[|u(\tilde{x},x_n')-u(\tilde{x},x_n)|\le 2|x_n'-x_n|,\]
as desired.

Let $x\in V$ and define $z=u(x)$. By definition, $E^z(\tilde{x})=x_n$ and $\nabla E^z(\tilde{x})=f(x)$. Therefore, there exists $\delta'>0$ such that if $|h|<\delta'$, then for each $i\in \{1,...,n-1\}$, $x+he_i\in V$ and $x_n(h)=E^z(\tilde{x}+h\tilde{e}_i)\in [x_n-(|f_i(x)|+1)|h|,x_n+(|f_i(x)|+1)|h|]$. Let $x(h)=(\tilde{x}+h\tilde{e}_i,x_n(h))$. By definition, $u(x(h))=z$. Because $u$ is Lipschitz in $x_n$, there exists a constant $L'>0$ such that for any $x\in V$ and $h\in (-\delta',\delta')$,
\[|u(x+he_i)-u(x)|=|u(x+he_i)-u(x(h))|\le L'(\max_{y\in \bar{V}}\|f(y)\|+1)|h|,\]
which implies that $u$ is Lipschitz on $V$.

Because $u$ is increasing in $x_n$, it has no local minimum/maximum point. Suppose that $u$ is differentiable at $x\in V$, and $\nabla u(x)=0$. Then, there exists $h>0$ such that $(\tilde{x},x_n+h)\in V$ and
\[0<u(\tilde{x},x_n+h)-u(x)<L^{-1}h.\]
Define $z=u(x)$ and $z(h)=u(\tilde{x},x_n+h)$. Then,
\[h=x_n+h-x_n=c(1;\tilde{x},z(h))-c(1;\tilde{x},z)\le L(z(h)-z)<h,\]
which is a contradiction. Therefore, $\nabla u(x)\neq 0$.

To summarize these results, we obtain a normal solution $u:V\to \mathbb{R}$ to (\ref{TDE}) around $x^*$. Therefore, the first assertion of this theorem holds. Now, suppose that $g$ is $C^k$. Then, $f$ is also $C^k$. By the bootstrap argument, we can show that $E^z$ is $C^{k+1}$. Therefore, for any $z$, $u^{-1}(z)$ is an $n-1$ dimensional $C^{k+1}$ manifold. Choose $x\in V$. Then, for each $z\in (x_n^*-\varepsilon,x_n^*+\varepsilon)$,
\[u(x)=z\Leftrightarrow c(1;\tilde{x},z)=x_n.\]
By the above argument, $c$ is $C^k$ and $\frac{\partial c}{\partial z}(1;\tilde{x},z)>0$, and by the implicit function theorem, we have that $u$ is $C^k$. This completes the proof. $\blacksquare$
\end{proof}

Suppose that $g:U\to \mathbb{R}^n\setminus\{0\}$ is $C^k$ and satisfies Jacobi's integrability condition. By Theorem \ref{MAIN}, there exists a $C^k$ normal solution to (\ref{TDE}) that has a $C^{k+1}$ level set. By Proposition \ref{prop2}, any normal solution to (\ref{TDE}) has a $C^{k+1}$ level set. However, as mentioned in the introduction, we cannot show that there exists a $C^{k+1}$ normal solution to (\ref{TDE}).

\section{Application}
Sometimes, we want to consider some optimization problem of a normal solution $u$ to (\ref{TDE}). To solve some optimization problems, the convexity/concavity of the objective function is desirable. However, the requirement ``any normal solution $u$ to (\ref{TDE}) is convex'' is too strong. That is, if $\phi:\mathbb{R}\to \mathbb{R}$ is $C^1$ and $\phi'(x)>0$ for any $x$, then for any normal solution $u$ to (\ref{TDE}), $v=\phi\circ u$ is also a normal solution to (\ref{TDE}). Although $u$ is convex, $v$ might be non-convex. Therefore, this requirement is not effective for our purpose.

An alternative requirement is as follows: ``any normal solution to (\ref{TDE}) is quasi-convex''. In general, if $u$ is quasi-convex and $\phi$ is increasing, then $v=\phi\circ u$ is also quasi-convex, and thus the above problem vanishes in this requirement.

However, of course, this relaxation of the requirement reduces the available techniques. For example, since the usual subdifferential is not defined for quasi-convex functions, theorems like the KKT theorem using subderivative conditions cannot be verified. Readers may consider that the following requirement ``there exists a convex normal solution to (\ref{TDE})'' is more suitable. However, this condition is a little stronger than our requirement. In the appendix, we see an example of $g$ in which any normal solution to (\ref{TDE}) on $U$ is quasi-convex, but there is no convex normal solution to (\ref{TDE}) defined on $U$.

In fact, any normal solution to (\ref{TDE}) is locally Lipschitz by definition, and for a locally Lipschitz function, the Clarke derivative can be defined, even if it is non-convex. Hence, the KKT theorem using the Clarke derivatives remains applicable. However, we present a slightly stronger result in which the Clarke derivative is not used. Thus, using the strong properties of convex functions is unnecessary, at least in this context.

We show the following theorem.

\begin{thm}\label{KKT}
Consider the following optimization problem:
\begin{align}
\min~~~~~&~u(x)\nonumber \\
\mathrm{subject\ to\ }&~x\in U,\nonumber \\
&~h_1(x)\le 0,\label{MP}\\
&~\vdots \nonumber \\
&~h_k(x)\le 0.\nonumber
\end{align}
Suppose that $U\subset \mathbb{R}^n$ is open and convex, $h_i$ is convex on $U$, and there exists a locally Lipschitz function $g:U\to \mathbb{R}^n\setminus \{0\}$ such that $u:U\to \mathbb{R}$ is a regular solution to \textup{(\ref{TDE})}. Moreover, suppose that there exists $x^+\in U$ such that $h_i(x^+)<0$ for all $i\in \{1,...,k\}$.\footnote{This condition is sometimes called {\bf Slater's condition}.} If $u$ is quasi-convex, then $x^*\in U$ is a solution to \textup{(\ref{MP})} if and only if it satisfies all constraints  $h_i(x^*)\le 0$ and there exists $\lambda_1,...,\lambda_k\ge 0$ such that\footnote{The former condition is called the {\bf Karush--Kuhn--Tucker condition}, or abbreviatedly, the {\bf KKT condition}. The latter condition is called the {\bf complementary slackness condition}.}
\[-g(x^*)\in \sum_{i=1}^k\lambda_i\partial h_i(x^*),\]
and for all $i\in \{1,...,k\}$,
\[\lambda_ih_i(x^*)=0.\]
\end{thm}

\begin{proof}
Define
\[C_1=\{x\in U|u(x)<u(x^*)\},\]
\[C_2=\{x\in U|\forall i\in \{1,...,k\},\ h_i(x)\le 0\}.\]
The lower contour set of $u$ at $x^*$ is defined as follows:
\[L(x^*)=\{x\in U|u(x)\le u(x^*)\}.\]
Because $x^*\in L(x^*)$, $L(x^*)$ is nonempty. Choose any $x\in L(x^*)$. Because $u$ is regular, $x-tg(x)\in C_1$ for any sufficiently small $t>0$. Therefore, $L(x^*)$ is the closure of $C_1$ and $C_1$ is the interior of $L(x^*)$. In particular, $C_1$ is nonempty, and because $u$ is quasi-convex, $C_1$ and $L(x^*)$ are convex. Since $x^+\in C_2$ and each $h_i$ is convex, $C_2$ is also nonempty and convex.

First, suppose that $x^*$ is a solution to (\ref{MP}). Then, $C_1\cap C_2=\emptyset$. By the separating hyperplane theorem, there exist $p^*\in \mathbb{R}^n\setminus\{0\}$ and $M\in \mathbb{R}$ such that
\[p^*\cdot z\le M\le p^*\cdot y\]
for all $y\in C_1$ and $z\in C_2$.

Because $x^*\in C_2$, $p^*\cdot x^*\le M$. Because $x^*$ is a limit point of $C_1$, $p^*\cdot x^*\ge M$. In conclusion, $p^*\cdot x^*=M$, and thus $x^*$ is a solution to the following convex problem:
\begin{align}
\min~~~~~&~-p^*\cdot z\nonumber \\
\mbox{subject to }&~z\in U,\nonumber \\
&~h_1(z)\le 0,\label{CONV}\\
&~\vdots \nonumber \\
&~h_k(z)\le 0.\nonumber
\end{align}
By the usual KKT theorem for convex programming, there exists $\mu_1,...,\mu_k\ge 0$ such that
\[p^*\in \sum_{i=1}^k\mu_i\partial h_i(x^*),\]
and for all $i\in \{1,...,k\}$,
\[\mu_ih_i(x^*)=0.\]
Choose any $C^1$ curve $c(t)$ such that $c(0)=x^*$ and $u(c(t))\equiv u(x^*)$. Then, $\dot{c}(0)\cdot g(x^*)=0$. Because each $c(t)$ is a limit point of $C_1$, we have that $p^*\cdot c(t)\ge M$. Therefore, $p^*\cdot \dot{c}(0)=0$. This implies that for $X=u^{-1}(u(x^*))$, both $p^*$ and $g(x^*)$ are orthogonal to $T_{x^*}(X)$, and thus $p^*$ is proportional to $g(x^*)$. Hence, $p^*=ag(x^*)$ for some $a\in\mathbb{R}\setminus \{0\}$. Choose any $t>0$ such that $x^*-tg(x^*)\in C_1$. Then,
\[p^*\cdot x^*-at\|p^*\|^2=p^*\cdot (x^*-tg(x^*))\ge M=p^*\cdot x^*,\]
which implies that $a<0$. Let $\lambda_i=-a^{-1}\mu_i$. Then, for all $i\in \{1,...,k\}$,
\[\lambda_ih_i(x^*)=0,\]
and
\[-g(x^*)=-a^{-1}p^*\in \sum_{i=1}^k\lambda_i\partial h_i(x^*),\]
as desired.

Next, suppose that $h_i(x^*)\le 0$ for all $i\in \{1,...,k\}$, and there exist $\lambda_1,...,\lambda_k\ge 0$ such that
\[-g(x^*)\in\sum_{i=1}^k\lambda_i\partial h_i(x^*)\]
and for all $i\in \{1,...,k\}$,
\[\lambda_ih_i(x^*)=0.\]
Let $p^*=-g(x^*)$, and consider the problem (\ref{CONV}). Then, by the usual KKT theorem for convex programming, $x^*$ is a solution to (\ref{CONV}). Hence, if we define $M=p^*\cdot x^*$, then $p^*\cdot z\le M$ for all $z\in C_2$. Next, applying the supporting hyperplane theorem for $L(x^*)$, there exists $q^*\in \mathbb{R}^n\setminus \{0\}$ such that $q^*\cdot x\ge q^*\cdot x^*$ for all $x\in L(x^*)$. Let $c(t)$ be a $C^1$ curve such that $c(0)=x^*$ and $u(c(t))=u(x^*)$ for all $t$. Then, $q^*\cdot c(t)\ge q^*\cdot x^*$ for all $t$, and thus $q^*\cdot \dot{c}(0)=0$. Hence, if we define $X=u^{-1}(u(x^*))$, then both $p^*$ and $q^*$ are orthogonal to $T_{x^*}(X)$, and thus there exists $a\in \mathbb{R}$ such that $a\neq 0$ and $p^*=aq^*$. For any $t>0$ such that $x^*+tp^*\in C_1$,
\[p^*\cdot (x^*+tp^*)=M+t\|p^*\|^2>M,\]
\[q^*\cdot (x^*+tp^*)\ge M.\]
Therefore, $a>0$ and for all $x\in L(x^*)$,
\[p^*\cdot x=aq^*\cdot x\ge aq^*\cdot x^*=p^*\cdot x^*=M.\]
Suppose that there exists $x\in C_1$ such that $p^*\cdot x=M$. Because $C_1$ is open, there exists $y\in C_1$ such that $p^*\cdot y<M$, which is a contradiction. To summarize our arguments, we obtain that
\[p^*\cdot z\le M<p^*\cdot y\]
for all $z\in C_2$ and $y\in C_1$, which implies that $x^*$ is a solution to (\ref{MP}). This completes the proof. $\blacksquare$
\end{proof}

The following theorem presents a method for judging whether a normal solution $u$ to (\ref{TDE}) is quasi-convex or not.

\begin{thm}\label{QC}
Suppose that $U\subset \mathbb{R}^n$ is open, $g:U\to \mathbb{R}^n\setminus \{0\}$ is locally Lipschitz and satisfies Jacobi's integrability condition. Then, the following three conditions are equivalent.
\begin{enumerate}[\upshape (i)]
\item for any $x^*\in U$, there exists an open neighborhood $V$ of $x^*$ such that for any $x,y\in V$,
\[g(x)\cdot y\ge g(x)\cdot x\Rightarrow g(y)\cdot x\le g(y)\cdot y.\]

\item For any $x\in U$ and $v\in \mathbb{R}^n$, if $g(x)\cdot v=0$, then
\[\limsup_{t\downarrow 0}\frac{1}{t}v\cdot g(x+tv)\ge 0.\]

\item Any regular solution $u$ to \textup{(\ref{TDE})} on a convex domain is quasi-convex.
\end{enumerate}
\end{thm}

\begin{proof}
Suppose that (i) holds, and choose any $x\in U$ and $v\in \mathbb{R}^n$ with $g(x)\cdot v=0$. Define $x(t)=x+tv$. Then, $g(x)\cdot x(t)=g(x)\cdot x$, and thus by (i), $g(x+tv)\cdot x\le g(x+tv)\cdot (x+tv)$ for any sufficiently small $t>0$. Therefore,
\begin{align*}
&~\limsup_{t\downarrow 0}\frac{1}{t}v\cdot g(x+tv)\\
=&~\limsup_{t\downarrow 0}\frac{1}{t^2}[(x+tv)-x]\cdot g(x+tv)\ge 0.
\end{align*}
Therefore, (ii) holds.

Next, suppose that (iii) holds. Choose $x^*\in U$, and a normal solution $u:V\to \mathbb{R}$ to (\ref{TDE}) around $x^*$. Without loss of generality, we can assume that $V$ is convex and $u$ is regular. Then, $u$ is quasi-convex. Choose $x\in V$, and consider the following minimization problem:
\begin{align*}
\min~~~~~&~u(y)\\
\mbox{subject to }&~y\in V,\\
&~g(x)\cdot y\ge g(x)\cdot x.
\end{align*}
Let
\[h_1(y)=g(x)\cdot (x-y).\]
Then, the condition $g(x)\cdot y\ge g(x)\cdot x$ is equivalent to $h_1(y)\le 0$, and $\partial h_1(x)=\{-g(x)\}$. By Theorem \ref{KKT}, $x$ is a solution to this problem, and thus $g(x)\cdot y\ge g(x)\cdot x$ implies that $u(y)\ge u(x)$. Suppose that $g(x)\cdot y\ge g(x)\cdot x$ but $g(y)\cdot x>g(y)\cdot y$. Then, $u(y)\ge u(x)$. Choose any $t>0$ such that $u(x-tg(x))<u(x)$ and $g(y)\cdot (x-tg(x))>g(y)\cdot y$. Then, $u(x)>u(x-tg(x))\ge u(y)$, which is a contradiction. Therefore, $g(y)\cdot x\le g(y)\cdot y$, and thus (i) holds.

It suffices to show that (ii) implies (iii). Suppose that (ii) holds but there exist a regular solution $u:V\to \mathbb{R}$ to (\ref{TDE}), $x,y\in V$, and $t\in [0,1]$ such that $V$ is convex and
\[u(y)\le u(x)<u((1-t)x+ty).\]
Let $x(t)=(1-t)x+ty$, $t^*=\max\arg\max\{u(x(t))|t\in [0,1]\}$ and $z^*=x(t^*)$. Then, $0<t^*<1$, $u(x(t))\le u(z^*)$ if $t\in [0,1]$, and $u(x(t))<u(z^*)$ if $t^*<t\le 1$. Define $p^*=g(z^*)$, $v=y-x$, and $V^*=\mbox{span}\{p^*,v\}$. If $v$ is proportional to $p^*$, then by Lemma \ref{LEMM}, $u(x(t))$ is monotone around $z^*$, which contradicts our assumption. Therefore, $V^*$ is $2$-dimensional.

Choose an open neighborhood $Z$ of $z^*$ such that $z\in Z$ implies $g(z)\cdot p^*>0$. Let $X=u^{-1}(u(z^*))\cap Z$ and $Y=X\cap (z^*+V^*)$. If $z\in X$, then $p^*\notin T_z(X)$, and thus $T_z(X)+V^*=\mathbb{R}^n$. Therefore, $X$ is transversal to $z^*+V^*$, and thus $Y$ is a $1$-dimensional $C^1$ manifold and there exists a $C^1$ local parametrization $y(s)$ of $Y$ such that $y(0)=z^*$. Let $\Phi(s,a)=y(s)+ap^*$ and $\Psi(b,c)=z^*+bv+cp^*$. Then, $\Psi$ is a $C^{\infty}$ diffeomorphism between $\mathbb{R}^2$ and $z^*+V^*$, and $D(\Psi^{-1}\circ \Phi)(0,0)$ is regular. By the inverse function theorem, there exist two open neighborhoods $U_1,V_1$ of $(0,0)$ such that $\Psi^{-1}\circ \Phi:U_1\to V_1$ is a $C^1$ diffeomorphism. Hence, $\Phi=\Psi\circ (\Psi^{-1}\circ \Phi)$ is a $C^1$ bijection on $U_1$ such that $\Phi^{-1}=(\Psi^{-1}\circ \Phi)^{-1}\circ \Psi^{-1}$ is also $C^1$ on $W_1\equiv \Phi(U_1)$. Because $g$ is locally Lipschitz, we can assume that $g$ is $L$-Lipschitz on $W_1$. Without loss of generality, we can assume that $U_1$ is convex and $W_1\subset Z$. By Lemma \ref{LEMM}, $u(z+\tau p^*)$ is increasing in $\tau$ for any $z\in W_1$. 

Let $(w(t),a(t))=\Phi^{-1}(x(t))$. Because $W_1,U_1$ are open and $(w(t^*),a(t^*))=\Phi^{-1}(z^*)=(0,0)\in U_1$, there exists $\delta>0$ such that if $|t-t^*|<\delta$, then $(w(t),0), (w(t),a(t))\in U_1$. Let
\[z(t)=\Phi(w(t),0)=x(t)-a(t)p^*,\ b(t)=\dot{z}(t)=v-\dot{a}(t)p^*.\]
By definition, $x(t)=y(w(t))+a(t)p^*$, and thus $z(t)=y(w(t))$. Because $u(y(w(t)))=u(z^*)$, $u(x(t))\le u(z^*)$, and $u(y(s)+\tau p^*)$ is increasing in $\tau$, $|t-t^*|<\delta$ implies that $a(t)\le 0$. Because $a(t^*)=0$, $t^*$ is a local maximum point of $a(t)$, and thus $\dot{a}(t^*)=0$. Moreover, by definition of $y(s)$, $\dot{y}(w(t))\cdot g(z(t))=0$, and thus $b(t)\cdot g(z(t))=0$.

Choose any $t\in [0,1]$ with $0<t-t^*<\delta$, and define $z^*(t,s)=z(t)+b(t)(s-t)$. Then,
\begin{align*}
&~\left|\frac{(b(s)-b(t))\cdot (g(z(s))-g(z(t)))}{s-t}\right|\\
\le&~|\dot{a}(s)-\dot{a}(t)|\|p^*\|L\frac{\|z(s)-z(t)\|}{|s-t|}\to 0\mbox{ as }s\to t,
\end{align*}
and
\begin{align*}
&~\left|\frac{b(t)\cdot (g(z(s))-g(z^*(t,s)))}{s-t}\right|\\
\le&~L\|b(t)\|\|p^*\|\frac{|a(s)-a(t)-\dot{a}(t)(s-t)|}{|s-t|}\to 0\mbox{ as }s\to t.
\end{align*}
Therefore,
\begin{align*}
&~(p^*\cdot g(z(t)))\liminf_{s\downarrow t}\frac{\dot{a}(t)-\dot{a}(s)}{s-t}\\
=&~\liminf_{s\downarrow t}\frac{(b(s)-b(t))\cdot g(z(t))}{s-t}=\liminf_{s\downarrow t}\frac{b(s)\cdot g(z(t))}{s-t}\\
=&~-\limsup_{s\downarrow t}\frac{b(s)\cdot (g(z(s))-g(z(t)))}{s-t}\\
=&~-\limsup_{s\downarrow t}\frac{1}{s-t}[(b(s)-b(t))\cdot (g(z(s))-g(z(t)))\\
&~+b(t)\cdot (g(z^*(t,s))-g(z(t)))+b(t)\cdot (g(z(s))-g(z^*(t,s)))]\le 0,
\end{align*}
where the last inequality follows from (ii). Therefore,
\[\liminf_{s\downarrow t}\frac{\dot{a}(t)-\dot{a}(s)}{s-t}\le 0.\]
By the symmetrical arguments, we can show that
\[\liminf_{s\uparrow t}\frac{\dot{a}(t)-\dot{a}(s)}{s-t}\le 0.\]
Now, choose $t_1\in [0,1]$ such that $0<t_1-t^*<\delta$, and define $h(t)=\dot{a}(t)(t_1-t^*)-\dot{a}(t_1)(t-t^*)$. Then, $h(t^*)=h(t_1)=0$, and thus, there exists $t^+\in (t^*,t_1)$ such that $h(t^+)$ attains either the maximum or the minimum on $[t^*,t_1]$. If $h(t^+)$ attains the maximum, then
\[0\le \liminf_{t\downarrow t^+}\frac{h(t^+)-h(t)}{t-t^+}=(t_1-t^*)\liminf_{t\downarrow t^+}\frac{\dot{a}(t^+)-\dot{a}(t)}{t-t^+}+\dot{a}(t_1)\le \dot{a}(t_1),\]
which implies that $\dot{a}(t_1)\ge 0$. If $h(t^+)$ attains the maximum, then by the symmetrical argument, we can also show that $\dot{a}(t_1)\ge 0$. Therefore, $\dot{a}(t)\ge 0$ if $1\ge t>t^*$ and $|t-t^*|<\delta$. This implies that if $t\in [0,1]$ and $|t-t^*|<\delta$, then $a(t)\ge a(t^*)=0$, and thus, $a(t)\equiv 0$. Hence, $u(x(t))=u(z^*)$ for all $t>t^*$ such that $t-t^*$ is sufficiently small, which is a contradiction. Thus, we have that $u$ is quasi-convex, as desired. This completes the proof. $\blacksquare$
\end{proof}

Sometimes, we need a stronger assumption than quasi-convexity. Suppose that $V\subset \mathbb{R}^n$ is convex. A function $u:V\to \mathbb{R}$ is said to be {\bf strictly quasi-convex} if, for any $x,y\in V$ and $t\in (0,1)$, if $x\neq y$, then
\[u((1-t)x+ty)<\max\{u(x),u(y)\}.\]
It is well-known that, if $u$ is strictly quasi-convex, then a solution to (\ref{MP}) is unique. Therefore, this assumption is useful.

The following proposition presents sufficient conditions for the strict quasi-convexity.

\begin{prop}\label{QC2}
Suppose that $U\subset \mathbb{R}^n$ is open, and $g:U\to \mathbb{R}^n\setminus \{0\}$ is locally Lipschitz and satisfies Jacobi's integrability condition. Consider the following three conditions.
\begin{enumerate}[\upshape (I)]
\item for any $x^*\in U$, there exists an open neighborhood $V$ of $x^*$ such that for any $x,y\in V$ with $x\neq y$,
\[g(x)\cdot y\ge g(x)\cdot x\Rightarrow g(y)\cdot x<g(y)\cdot y.\]

\item For any $x\in U$ and $v\in \mathbb{R}^n$, if $g(x)\cdot v=0$ and $v\neq 0$, then
\[\limsup_{t\downarrow 0}\frac{1}{t}v\cdot g(x+tv)>0.\]

\item Any regular solution $u$ to \textup{(\ref{TDE})} defined on a convex domain is strictly quasi-convex.
\end{enumerate}
Then, \textup{(I)} is equivalent to \textup{(III)}, and \textup{(II)} implies both \textup{(I)} and \textup{(III)}.
\end{prop}

\begin{proof}
First, suppose that $u:V\to \mathbb{R}$ is a regular quasi-concave solution to (\ref{TDE}). Choose $x\in V$, and consider the following minimization problem:
\begin{align*}
\min~~~~~&~u(y)\\
\mbox{subject to }&~y\in V,\\
&~g(x)\cdot y\ge g(x)\cdot x.
\end{align*}
By the same arguments as in the proof of Theorem \ref{QC}, we have that $x$ is a solution to this minimization problem. In particular, if $g(x)\cdot y\ge g(x)\cdot x$, then $u(y)\ge u(x)$. Suppose, in addition, that $u$ is strictly quasi-convex. If there is another solution $y\in V$ to the above problem, then for $z=(1-2^{-1})x+2^{-1}y$, $g(x)\cdot z\ge g(x)\cdot x$ and $u(z)<u(x)$, which is a contradiction. Therefore, $x$ is the unique solution to the above problem, and thus if $g(x)\cdot y\ge g(x)\cdot x$ and $y\neq x$, then $u(x)<u(y)$.

Now, suppose that (III) holds. Choose any $x^*\in U$, and a regular solution $u:V\to \mathbb{R}$ to (\ref{TDE}) such that $V$ is convex. Then, $u$ is strictly quasi-convex. If $y\neq x$ and $g(x)\cdot y\ge g(x)\cdot x$, then by the above arguments, we have $u(x)<u(y)$, and thus $g(y)\cdot x<g(y)\cdot y$. Hence, (I) holds.

Next, suppose that (I) holds. Choose a regular solution $u:V\to \mathbb{R}$ to (\ref{TDE}), where $V$ is convex. In this case, (i) of Theorem \ref{QC} holds, and thus $u$ is quasi-convex. Hence, $g(x)\cdot y\ge g(x)\cdot x$ implies that $u(y)\ge u(x)$. Suppose that $g(x)\cdot y\ge g(x)\cdot x$, $y\neq x$, and $u(y)=u(x)$. If $g(x)\cdot y>g(x)\cdot x$, then there exists $s>0$ such that for $w=y-sg(y)$, $g(x)\cdot w>g(x)\cdot x$. Then, $u(y)>u(w)\ge u(x)=u(y)$, which is a contradiction. Therefore, $g(x)\cdot y=g(x)\cdot x$. Choose any $t\in (0,1)$ and let $z=(1-t)x+ty$. Then, $g(x)\cdot z=g(x)\cdot x$, and thus $u(z)\ge u(x)$. Because $u$ is quasi-convex, $u(z)=u(x)=u(y)$. By (I), $g(z)\cdot x<g(z)\cdot z$. Therefore, $g(z)\cdot y>g(z)\cdot z$, and thus $u(y)>u(z)$, which is a contradiction. Hence, if $g(x)\cdot y\ge g(x)\cdot x$ and $y\neq x$, then $u(y)>u(x)$.

Now, suppose that $u$ is not strictly quasi-convex. Then, there exist $x,y\in V$ and $t\in (0,1)$ such that $x\neq y$ and
\[u((1-t)x+ty)\ge\max\{u(x),u(y)\}.\]
Let $z=(1-t)x+ty$. Because $u(z)\ge u(x)$ and $u(z)\ge u(y)$, we have that $g(z)\cdot x\le g(z)\cdot z$ and $g(z)\cdot y\le g(z)\cdot z$. Because
\[g(z)\cdot z=(1-t)(g(z)\cdot x)+t(g(z)\cdot y),\]
we have that $g(z)\cdot x=g(z)\cdot y=g(z)\cdot z$. Hence, $u(z)>u(x)$ and $u(z)>u(y)$, which contradicts the quasi-convexity of $u$. Therefore, $u$ must be strictly quasi-convex, and (III) holds.

It suffices to show that (II) implies (I). Suppose that (II) holds, and choose $x^*\in U$ and a regular solution $u:V\to \mathbb{R}$ to (\ref{TDE}) around $x^*$, where $V$ is convex. Because (II) implies (ii) of Theorem (\ref{QC}), $u$ is quasi-convex, and (i) of Theorem \ref{QC} holds. Hence, $g(x)\cdot y\ge g(x)\cdot x$ implies that $u(y)\ge u(x)$.

Suppose that there exist $x,y\in V$ such that $x\neq y$, $g(x)\cdot y\ge g(x)\cdot x$, and $g(y)\cdot x\ge g(y)\cdot y$. By (i), we have that $g(x)\cdot y=g(x)\cdot x$ and $g(y)\cdot x=g(y)\cdot y$. Let $v=y-x$ and $x(t)=x+tv$. Choose any $t\in [0,1]$. Because $g(x)\cdot x(t)=g(x)\cdot x$, $u(x(t))\ge u(x)$. In particular, $u(y)\ge u(x)$. By the same arguments, we have that $u(x(t))\ge u(y)$, and in particular, $u(x)\ge u(y)$. Therefore, $u(x)=u(y)$. Because $u$ is quasi-convex, $u(x(t))=u(x)=u(y)$. Therefore, if $X=u^{-1}(u(x))$, then $v\in T_{x(t)}(X)$, and thus $g(x(t))\cdot v=0$ for all $t\in [0,1]$. This implies that
\[\limsup_{t\downarrow 0}\frac{1}{t}v\cdot g(x(t))=0,\]
which contradicts (II). Therefore, such $x,y$ are absent, and (I) holds. This completes the proof. $\blacksquare$
\end{proof}

In this proposition, (I) does not imply (II). We provide a counterexample in the appendix.

\appendix
\section{Examples}
\subsection{The Example in the Introduction}
In this subsection, we discuss an example presented by Debreu \cite{D76}.

Recall the definition of $g$:
\[g(x_1,x_2)=\begin{cases}
\left(\frac{x_2^2}{\sqrt{1+x_2^4}},\frac{1}{\sqrt{1+x_2^4}}\right) & \mbox{if } x_2\ge 0,\\
(0,1) & \mbox{if }x_2<0.
\end{cases}\]
A normal solution $u$ to (\ref{TDE}) around $(0,0)$ is as follows:
\[u(x_1,x_2)=\begin{cases}
\frac{x_2}{1-x_1x_2} & \mbox{if }x_2\ge 0,\\
x_2 & \mbox{if } x_2<0.
\end{cases}\]
This function $u$ is not $C^2$, but only $C^1$. Suppose that there exists a normal solution $v$ to (\ref{TDE}) around $(0,0)$ such that $v$ is $C^2$. Without loss of generality, we can assume that the domain $V$ of $v$ is $(-\varepsilon,\varepsilon)^2$, where $\varepsilon>0$ is sufficiently small, and $v$ is regular. Therefore, we can assume that $u$ is defined on $V$, and by Corollary \ref{COR2}, there exists an increasing $C^1$ function $\phi$ such that $v=\phi\circ u$. Note that, because $u(0,x_2)=x_2$, $\phi(c)=v(0,c)$. By Proposition \ref{prop1}, there exists a $C^1$ positive function $\mu:V\to \mathbb{R}$ such that $\nabla v=\mu g$ on $V$. Therefore, we have that
\[\phi'(c)=\frac{\partial v}{\partial x_2}(0,c)=\mu(0,c)g_2(0,c)>0\]
for all $c\in (-\varepsilon,\varepsilon)$. By the chain-rule,
\[\frac{\frac{\partial v}{\partial x_2}(x)}{\phi'(u(x))}=\frac{\partial u}{\partial x_2}(x).\]
However, if $x_2=0$ and $x_1\neq 0$, then the right-hand side is not differentiable in $x_2$ at $x$, whereas the left-hand side is differentiable. This is a contradiction. Therefore, such a $v$ is absent.

\subsection{An Example of Non-Convex Solutions}
Recall the definition of convexity and quasi-convexity. Suppose that $U\subset \mathbb{R}^n$ is convex, and consider a function $f:U\to \mathbb{R}$. Then, $f$ is said to be {\bf convex} if, for any $x,y\in U$ and $t\in [0,1]$,
\[f((1-t)x+ty)\le (1-t)f(x)+tf(y).\]
Next, $f$ is said to be {\bf quasi-convex} if, for any $x,y\in U$ and $t\in [0,1]$,
\[f((1-t)x+ty)\le \max\{f(x),f(y)\}.\]
Because
\[(1-t)f(x)+tf(y)\le \max\{f(x),f(y)\},\]
we have that any convex function is quasi-convex. Suppose that $\phi:f(U)\to \mathbb{R}$ is increasing, and define $h=\phi\circ f$. If $f$ is quasi-convex, then
\[h((1-t)x+ty)=\phi(f((1-t)x+ty))\le \phi(\max\{f(x),f(y)\})=\max\{h(x),h(y)\},\]
which implies that $h$ is also quasi-convex. However, $h$ might be non-convex even if $f$ is convex: consider $f(x)=x^3$ on $\mathbb{R}_+$ and $\phi(x)=x^{1/4}$.

If there is a convex normal solution to (\ref{TDE}) defined on some convex set $U$, then it is quasi-convex. If $U=\mathbb{R}^n_{++}$ and $g(x)\gg 0$ for any $x\in U$, then by Corollary \ref{COR2}, {\bf any normal solution} to (\ref{TDE}) defined on $U$ is either quasi-convex or quasi-concave. The problem considered in this subsection is the converse of this problem. Suppose that any normal solution to (\ref{TDE}) on $U=\mathbb{R}^n_{++}$ is either quasi-convex or quasi-concave. Is there a convex normal solution? The answer is negative. The following counterexample was obtained by Arrow and Enthoven \cite{AE}.

Choose $U=\mathbb{R}^2_{++}$. We consider the following function:
\[g_1(x_1,x_2)=1+\frac{x_1+1}{\sqrt{(x_1+1)^2+4x_2}},\]
\[g_2(x_1,x_2)=\frac{2}{\sqrt{(x_1+1)^2+4x_2}}.\]
A normal solution to (\ref{TDE}) is as follows.
\[u(x)=(x_1-1)+\sqrt{(x_1+1)^2+4x_2}.\]
This function is $C^1$ and strictly increasing. It can be easily checked that the range of $u$ is $\mathbb{R}_{++}$. For any $C>0$, if $x_1,x_2>0$ and 
\[(C-(x_1-1))^2=(x_1+1)^2+4x_2,\]
Then
\[0<4x_2=(C-2x_1)(C+2),\]
and thus, $C>2x_1>x_1-1$. Therefore,
\begin{align*}
u(x)=C\Leftrightarrow&~C-(x_1-1)=\sqrt{(x_1+1)^2+4x_2}\\
\Leftrightarrow&~(C-(x_1-1))^2=(x_1+1)^2+4x_2\\
\Leftrightarrow&~4x_2=C^2-2C(x_1-1)-4x_1\\
\Leftrightarrow&~x_2=(-1-C/2)x_1+(C^2+2C)/4.
\end{align*}
Hence, any level set of $u$ is a straight line. Suppose that $u(x)\ge u(y)$ and $u((1-t)x+ty)>u(x)$ for some $t\in [0,1]$. Then, $0<t<1$. By the intermediate value theorem, there exists $t^*\in (t,1]$ such that $u((1-t^*)x+t^*y)=u(x)$. Because the level set $u^{-1}(u(x))$ is a straight line, we have that $u((1-t)x+ty)=u(x)$, which is a contradiction. Therefore, we have that $u$ is quasi-convex. By Corollary \ref{COR2}, any normal solution to (\ref{TDE}) is either quasi-convex or quasi-concave.

Next, define
\[\psi_1(C)=\frac{C}{2},\ \psi_2(C)=\frac{C^2+2C}{4}.\]
Then,
\begin{align*}
\frac{x_1}{\psi_1(C)}+\frac{x_2}{\psi_2(C)}=1\Leftrightarrow&~\frac{2Cx_1+4x_1+4x_2}{C^2+2C}=1\\
\Leftrightarrow&~2Cx_1+4x_1+4x_2=C^2+2C\\
\Leftrightarrow&~x_2=(-1-C/2)x_1+(C^2+2C)/4\Leftrightarrow u(x)=C.
\end{align*}
Choose any $k>0$ and define
\[h_k(C)=\sup\{x_1+kx_2|u(x)\le C\},\]
\[h_k^*(C)=\sup\{-x_1-kx_2|u(x)\ge C\}.\]
Using the intermediate value theorem appropriately, we have that
\begin{align*}
h_k(C)=&~\sup\{x_1+kx_2|u(x)\le C\}=\sup\{x_1+kx_2|u(x)=C\}\\
=&~\sup\left\{x_1+kx_2\left|\frac{x_1}{\psi_1(C)}+\frac{x_2}{\psi_2(C)}=1\right.\right\}\\
=&~\max\{\psi_1(C),k\psi_2(C)\},
\end{align*}
and
\begin{align*}
h_k^*(C)=&~\sup\{-x_1-kx_2|u(x)\ge C\}=\sup\{-x_1-kx_2|u(x)=C\}\\
=&~-\inf\left\{x_1+kx_2\left|\frac{x_1}{\psi_1(C)}+\frac{x_2}{\psi_2(C)}=1\right.\right\}\\
=&~-\min\{\psi_1(C),k\psi_2(C)\}.
\end{align*}

Now, suppose that there exists a convex normal solution $w:U\to \mathbb{R}$ to (\ref{TDE}). By Corollary \ref{COR2}, it is either strictly increasing or strictly decreasing, and there is a monotone function $\varphi:I\to \mathbb{R}$ such that $I=w(U)$ and $u=\varphi\circ w$. Because $w$ is convex and the domain $U$ is open, it is locally Lipschitz, and thus $I$ is an interval.\footnote{Any convex function defined on an open and convex set is locally Lipschitz. See, for example, Rockafellar \cite{R}.} Clearly, $\varphi(z)>0$ for all $z\in I$.

We divide the proof into two cases.

\vspace{12pt}
\noindent
{\bf Case 1}. Suppose that $\varphi$ is increasing. Then, $w$ is convex and strictly increasing. Choose $z_0,z_1\in I$ and $t\in [0,1]$, and define $z_t=(1-t)z_0+tz_1$. Choose any $\varepsilon>0$. Let $g_k=h_k\circ \varphi$. For $i\in \{0,1\}$, there exists $x^i\in \mathbb{R}^2_{++}$ such that
\[g_k(z_i)-\varepsilon<x^i_1+kx^i_2,\ u(x^i)\le \varphi(z_i).\]
Hence, $w(x^i)\le z_i$, and thus
\[w((1-t)x^0+tx^1)\le (1-t)w(x^0)+tw(x^1)\le z_t,\]
which implies that $u((1-t)x^0+tx^1)\le \varphi(z_t)$. Therefore,
\begin{align*}
g_k(z_t)\ge&~[(1-t)x^0_1+tx^1_1]+k[(1-t)x^0_2+tx^1_2]\\
=&~(1-t)(x^0_1+kx^0_2)+t(x^1_1+kx^1_2)\\
>&~(1-t)g_k(z_0)+tg_k(z_1)-\varepsilon.
\end{align*}
Since $\varepsilon>0$ is arbitrary,
\[g_k(z_t)\ge (1-t)g_k(z_0)+tg_k(z_1).\]
Hence, $g_k$ is concave. Letting $k\to 0$, we have that $\psi_1\circ \varphi$ is concave. Because $\psi_1$ is increasing and linear, $\varphi$ itself is concave.

Therefore, $\varphi$ is locally Lipschitz. By Rademacher's theorem, there exists $z\in I$ such that $\varphi$ is differentiable at $z$. If $\varphi'(z)=0$, then $z$ is a maximum point of $\varphi$, which contradicts the monotonicity assumption of $\varphi$. Therefore, $\varphi'(z)>0$. Define $k=\frac{2}{\varphi(z)+2}$. Then, $\psi_1(\varphi(z))=k\psi_2(\varphi(z))$. By a simple calculation,
\begin{align*}
\left.\frac{d}{dz'}(\psi_1(\varphi(z'))-k\psi_2(\varphi(z')))\right|_{z'=z}=&~\frac{\varphi'(z)}{2}-k\psi_2'(\varphi(z))\varphi'(z)\\
=&~\varphi'(z)\left(\frac{1}{2}-\frac{\varphi(z)+1}{\varphi(z)+2}\right)<0.
\end{align*}
Hence, there exists a neighborhood $I'$ of $z$ such that if $z'\in I'$ and $z'>z$, then $g_k(z')=k\psi_2(\varphi(z'))$, and if $z'\in I'$ and $z>z'$, then $g_k(z')=\psi_1(\varphi(z'))$. Because $g_k$ is concave, we have that
\[k(\psi_2\circ \varphi)'(z)=D_+g_k(z)\le D_-g_k(z)=\frac{\varphi'(z)}{2},\]
and thus,
\[\varphi'(z)<2k\psi_2'(\varphi(z))\varphi'(z)\le \varphi'(z),\]
which is a contradiction.

\vspace{12pt}
\noindent
{\bf Case 2}. Suppose that $\varphi$ is decreasing. Define $w^*=-w$ and $\varphi^*(z)=\varphi(-z)$. Then, $w^*$ is concave and strictly increasing, $\varphi^*$ is increasing, and $u=\varphi^*\circ w^*$. Choose $z_0,z_1\in -I$ and $t\in [0,1]$, and define $z_t=(1-t)z_0+tz_1$. Choose any $\varepsilon>0$. Let $g_k^*=h_k^*\circ \varphi^*$. For $i\in \{0,1\}$, there exists $x^i\in \mathbb{R}^2_{++}$ such that
\[g_k^*(z_i)-\varepsilon<-x^i_1-kx^i_2,\ u(x^i)\ge \varphi^*(z_i).\]
Hence, $w^*(x^i)\ge z_i$, and thus,
\[w^*((1-t)x^0+tx^1)\ge (1-t)w^*(x^0)+tw^*(x^1)\ge z_t,\]
which implies that $u((1-t)x^0+tx^1)\ge \varphi^*(z_t)$. Therefore,
\begin{align*}
g_k^*(z_t)\ge&~-[(1-t)x^0_1+tx^1_1]-k[(1-t)x^0_2+tx^1_2]\\
=&~(1-t)(-x^0_1-kx^0_2)+t(-x^1_1-kx^1_2)\\
>&~(1-t)g_k^*(z_0)+tg_k^*(z_1)-\varepsilon.
\end{align*}
Since $\varepsilon>0$ is arbitrary,
\[g_k^*(z_t)\ge (1-t)g_k^*(z_0)+tg_k^*(z_1).\]
This implies that $g_k^*$ is concave. Letting $k\to \infty$, we have that $-\psi_1\circ \varphi^*$ is concave. Because $\psi_1$ is increasing and linear, we have that $\varphi^*$ is convex, and hence $\varphi$ is also convex.

Hence, $\varphi$ is locally Lipschitz. By Rademacher's theorem, there exists $z\in I$ such that $\varphi$ is differentiable at $z$. If $\varphi'(z)=0$, then $z$ is a minimum point of $\varphi$, which contradicts the monotonicity assumption of $\varphi$. Therefore, $\varphi'(z)<0$. Define $k=\frac{2}{\varphi(z)+2}$. Then, $\psi_1(\varphi(z))=k\psi_2(\varphi(z))$. By a simple calculation,
\begin{align*}
\left.\frac{d}{dz'}(\psi_1(\varphi(z'))-k\psi_2(\varphi(z')))\right|_{z'=z}=&~\frac{\varphi'(z)}{2}-k\psi_2'(\varphi(z))\varphi'(z)\\
=&~\varphi'(z)\left(\frac{1}{2}-\frac{\varphi(z)+1}{\varphi(z)+2}\right)>0.
\end{align*}
Hence, there exists a neighborhood $I'$ of $z$ such that if $z'\in I'$ and $z'>z$, then $g_k^*(z')=-k\psi_2(\varphi(z'))$, and if $z'\in I'$ and $z>z'$, then $g_k^*(z')=-\psi_1(\varphi(z'))$. Because $g_k^*$ is concave, we have that
\[-k(\psi_2\circ \varphi)'(z)=D_+g_k^*(z)\le D_-g_k^*(z)=-\frac{\varphi'(z)}{2},\]
and thus,
\[\varphi'(z)>2k\psi_2'(\varphi(z))\varphi'(z)\ge \varphi'(z),\]
which is a contradiction.

Hence, in each case, we lead a contradiction. This implies that such a convex function $w$ does not exist.

In this proof, the idea of using an artificial function $h_k$ is obtained by Fenchel \cite{F}. For more detailed arguments, see Kannai \cite{K} or Monteiro \cite{MO}.

\subsection{Katzner's Counterexample}
In this subsection, we present an example of $g$ in which (III) of Proposition \ref{QC2} holds but (II) is violated. This example was found by Katzner \cite{KA}.

Consider $U=\mathbb{R}^2_{++}$, and
\[g_1(x)=-3x_1^2x_2-x_2^3,\]
\[g_2(x)=-x_1^3-3x_1x_2^2.\]
A regular solution to (\ref{TDE}) is as follows.
\[u(x)=-(x_1^3x_2+x_1x_2^3).\]
It is easy to show that $u$ is strictly quasi-convex. By simple calculation,
\[\frac{\partial g_1}{\partial x_1}(x)=-6x_1x_2,\ \frac{\partial g_1}{\partial x_2}(x)=-3x_1^2-3x_2^2,\]
\[\frac{\partial g_2}{\partial x_1}(x)=-3x_1^2-3x_2^2,\ \frac{\partial g_2}{\partial x_2}(x)=-6x_1x_2.\]
Choose $v=(-1,1)$ and $x=(1,1)$. Then $g(x)=(-4,-4)$, and thus $g(x)\cdot v=0$. However,
\begin{align*}
\lim_{t\downarrow 0}\frac{1}{t}[v\cdot g(x+tv)]=&~\lim_{t\downarrow 0}\frac{1}{t}[v\cdot (g(x+tv)-g(x))]\\
=&~\sum_{i,j=1}^2\frac{\partial g_i}{\partial x_j}(x)v_iv_j\\
=&~0,
\end{align*}
and thus $g$ violates (II) of Proposition \ref{QC2}.

\section*{Acknowledgments}

This paper is supported by Chuo University Grant for Special Research. I am grateful to Toru Maruyama and Hisatoshi Tanaka for their helpful suggestions.

\end{document}